\documentclass[11pt,english]{extarticle}
\usepackage[T1]{fontenc}
\usepackage[latin9]{inputenc}
\usepackage{geometry}
\usepackage{float}
\usepackage{multirow}
\usepackage{amsmath}
\usepackage{amsthm}
\usepackage{amssymb}
\usepackage{stmaryrd}
\usepackage{graphicx}
\usepackage[numbers]{natbib}

\makeatletter

\providecommand{\tabularnewline}{\\}

\theoremstyle{plain}
    \ifx\thechapter\undefined
	    \newtheorem{thm}{\protect\theoremname}
	  \else
      \newtheorem{thm}{\protect\theoremname}[chapter]
    \fi
\theoremstyle{remark}
    \ifx\thechapter\undefined
      \newtheorem{rem}{\protect\remarkname}
    \else
      \newtheorem{rem}{\protect\remarkname}[chapter]
    \fi
\theoremstyle{definition}
    \ifx\thechapter\undefined
      \newtheorem{defn}{\protect\definitionname}
    \else
      \newtheorem{defn}{\protect\definitionname}[chapter]
    \fi

\usepackage{babel}

\providecommand{\theoremname}{Theorem}

\@ifundefined{showcaptionsetup}{}{%
 \PassOptionsToPackage{caption=false}{subfig}}
\usepackage{subfig}
\makeatother

\usepackage{babel}
\providecommand{\definitionname}{Definition}
\providecommand{\remarkname}{Remark}
\providecommand{\theoremname}{Theorem}

\begin{document}
\title{Nonlinear Model Reduction to Random Spectral Submanifolds in Random
Vibrations}
\author{Zhenwei Xu$^{1}$, Roshan S. Kaundinya$^{1}$, Shobhit Jain$^{2}$
and George Haller\thanks{Corresponding author. Email: georgehaller@ethz.ch}$^{1}$\\
 $^{1}$Institute for Mechanical Systems, ETH Z\"urich\\
 $^{2}$Delft Institute of Applied Mathematics, TU Delft \\
}
\maketitle
\begin{abstract}
Dynamical systems in engineering and physics are often subject to
irregular excitations that are best modeled as random. Monte Carlo
simulations are routinely performed on such random models to obtain
statistics on their long-term response. Such simulations, however,
are prohibitively expensive and time consuming for high-dimensional
nonlinear systems. Here we propose to decrease this numerical burden
significantly by reducing the full system to very low-dimensional,
attracting, random invariant manifolds in its phase space and performing
the Monte Carlo simulations on that reduced dynamical system. The
random spectral submanifolds (SSMs) we construct for this purpose
generalize the concept of SSMs from deterministic systems under uniformly
bounded random forcing. We illustrate the accuracy and speed of random
SSM reduction by computing the SSM-reduced power spectral density
of the randomly forced mechanical systems that range from simple oscillator
chains to finite-element models of beams and plates. 
\end{abstract}

\section{Introduction}

Excitation without temporally recurrent features is ubiquitous in
nature and hence many mechanical structures are subjected to it. Examples
include wind forcing on bridges, turbulent airflow over airplane wings
and seismic forces from earthquakes on buildings. This has been stimulating
significant interest in stochastic approaches to mechanics problems,
dating back to seminal contributions by \citet{caughey1962EquivalentLT}
and \citet{crandall_perturbation_1963} (see also \citet{crandall2014random},
\citet{caughey71}, \citet{wirsching06} and \citet{li2009stochastic}
for surveys).

Notable technical advances in the analysis of stochastically forced
systems are perturbative methods for higher-order statistics in weakly
nonlinear systems (see, e.g., \citet{caughey1962EquivalentLT} and
\citet{crandall_perturbation_1963}) and stochastic averaging for
systems with weak nonlinearity and weak damping (see \citet{roberts86}).
In these and other related methods, random vibrations have been modeled
either as random dynamical systems (ie., ordinary differential equations
appended with a white noise perturbation) or stochastic differential
equations (i.e., It\^o processes for the differential of the unknown
with a drift term and a Brownian motion term); see \citet{arnold03}
and \citet{duan15} for discussions of the technicalities involved
in both approaches. 

To a large extent, both of these modeling approaches are driven by
the convenience of having the Kolmogorov and the Fokker-Planck partial
differential equations (PDEs) at one's disposal for the first-order
statistics of the response. Both approaches, however, allow for the
occurrence of arbitrarily large forcing values with nonzero probability,
which is unrealistic in physical settings, as already pointed out
by \citet{caughey71}. A similar situation arises in other fields
with prominent randomness, such as financial modeling, in which stochastic
ODEs with Brownian motion are convenient but do not quite reproduce
practically observed financial time series (see, e.g., \citet{GHAOUI02}).
Even from a purely mathematical perspective, showing the existence
of the random equivalents of deterministic structures (such as invariant
manifolds) in the phase space under arbitrarily large perturbations
is a futile undertaking. Indeed, even simple structurally stable stable
dynamical features, such as asymptotically stable fixed points or
attracting limit cycles, will be destroyed by unbounded perturbations.

For these reasons, we consider here the physically more realistic
case of external random forcing whose realizations are uniformly bounded
in time. Under this assumption, one may reasonably seek random extensions
of dynamical structures that are known to enable mathematically rigorous
model reduction in deterministic vibratory systems. Such structures
are known in deterministic vibrations since the seminal work of \citet{shaw93}
on nonlinear normal modes (NNMs), defined as invariant manifolds tangent
to modal subspaces of mechanical systems (see also \citet{kerschen09},
\citet{mikhlin11} and \citet{mikhlin23} for reviews). 

More recent work by \citet{haller16} and \citet{haller23} revealed
that the envisioned deterministic NNMs only exist under specific nondegeneracy
and nonresonance conditions. Even under those conditions, NNMs are
never unique. Rather, a continuous family of spectral submanifolds
(SSMs) exists that are all tangent to any selected span of nonresonant
spectral subspaces. The smoothest such SSM admits a regular Taylor
expansion up to the degree of the dynamical system, whereas the remaining
secondary (or fractional) SSMs only admit fractional-powered expressions.
The general theory of SSMs has also been extended to cover fixed points
with instabilities (\citet{haller23}) and systems subject to uniformly
bounded temporally aperiodic forcing, including chaotic forcing (see
\citet{haller24}). 

These results on generalized deterministic SSMs have been implemented
in a continually expanding open-source MATLAB code, \textit{SSMTool,}
with a growing collection of nonlinear examples (see \citet{jain23}
for the latest version). The package can also handle parametric resonance
(see \citet{thurnher2023}) and algebraic constraints (see \citet{li22c}).
Matlab- and Python-based open source packages, \emph{SSMLearn} and
\emph{fastSSM}, for a purely data-driven extraction of SSMs have also
become available (see \citet{cenedese22a} and \citet{axas23}), broadening
the applicability of SSM-based model reduction to diverse fields,
such as control of soft robots (see \citet{alora23}), transition
problems in fluids (see \citet{Kaszas_Haller_2024}) and modeling
of fluid-structure interactions from videos (see \citet{xu24}). 

The idea of extending the Shaw--Pierre NNM concept to random vibrations
is already implicit in the work of \citet{worden17}. Motivated by
the full decoupling of linear systems in a modal basis, these authors
propose an alternative definition of deterministic NNMs as $n$ statistically
independent nonlinear mappings from a $2n$-dimensional phase space
to two-dimensional (2D) spaces. \citet{worden17} assume polynomial
expansions for these NNM mappings and seek to determine the coefficients
in these expansions from the response of the system under stochastic
forcing. Specifically, they obtain the polynomial coefficients by
minimizing the cross-correlations of envisioned mappings via regression.

Targeting deterministic NNMs via stochastic forcing, \citet{worden17}
implicitly assume that (a) an appropriately modified notion of an
invariant manifold continues to exist under stochastic forcing (b)
this manifold is in fact the same as the unforced manifold, irrespective
of the forcing level. While (a) cannot be guaranteed for the unbounded
noise employed in their approach and (b) is generally not the case,
the idea is clearly a novel and important step in exploring the fate
of NNMs under general forcing and in a data-driven setting. A follow-up
paper by \citet{Tsialiamanis22} on the same idea is critical of the
first implementation of the idea and uses instead generative adversarial
networks to learn the NNM mappings, still under stochastic forcing.
More recent work by \citet{simpson21} explores this idea further
by employing long-short term neural networks in the procedure. The
dimension of data-driven applications in these advances remains in
the tens of degrees of freedom. No explicit reduced-models are obtained
in the end, as the approach seeks only the mapping from the data to
reduced coordinates and back. 

In contrast, SSMLearn has been successfully employed to extract predictive
deterministic SSM-reduced models from experimental measurements (see
\citet{cenedese22b}), experimental videos of continua (of infinite
degrees of freedom, see \citet{cenedese22a} and \citet{yang24})
and numerical data from finite-element codes exceeding a million degrees
of freedom (see \citet{cenedese24}). For this reason, we do not seek
here a further improvement of data-driven deterministic SSM extraction.
Rather, we focus on what the above cited statistical approaches implicitly
assume: the existence of random SSMs in stochastically forced deterministic
systems. We show that a combination of prior results on a class of
random invariant manifolds and recent results on deterministic invariant
manifolds can be used to infer the existence of random SSMs and their
reduced dynamics with full mathematical rigor.

This paper is organized as follows. We first prove the existence of
random spectral submanifolds (random SSMs) for randomly forced dynamical
systems. These SSMs turn out to exist under smoothness and nonresonance
conditions on the deterministic unforced system and under a uniform
boundedness assumption on the random forcing. We also derive leading-order
formulas for computing reduced models on random SSMs. We then demonstrate
the performance of random SSM reduction in four examples: a suspension
system also known as a quarter-car model moving on an irregular road,
a building model subject to earthquakes, a von-Kármán beam undergoing
random base excitation and a rectangular von-Kármán plate placed in
a windy medium. We conclude by summarizing our contributions and discussing
future work.

\section{SSMs in randomly forced systems}

\subsection{Setup}

We consider small random perturbations of a deterministic set of first-order
nonlinear ordinary differential equations (ODEs) with an attracting
fixed point. The randomness of the process is introduced via the dependence
of the random perturbation on elements $\boldsymbol{\nu}$ of a probability
space $\mathcal{V}$, which is endowed with a probability measure
$\mathbb{P}$ and a corresponding sigma algebra $\mathcal{F}$. The
evolving randomness along trajectories will be induced by a metric
dynamical system $\boldsymbol{\theta}^{t}\colon\mathcal{V}\to\mathcal{V}$
that creates a time-dependent random variable $\boldsymbol{\theta}^{t}\left(\boldsymbol{\nu}\right)$
on which the random differential equation depends in a smooth and
uniformly bounded fashion.\\

This smooth dependence will guarantee uniform boundedness for the
noise in time, which is a realistic assumption for applications in
mechanics. In contrast, randomness modeled by classic stochastic differential
equations (or It\^o processes) can display arbitrarily large changes
in the noise over short periods of time. Such unbounded perturbations
would be unrealistic in our context of real-life mechanical systems
subject to noisy excitation. This was already pointed out in early
work on random vibrations by \citet{caughey71}, who observed that
unbounded noise is a convenient assumption that is unsupported by
physics but yields simpler statistics. We do not suggest that this
convenience assumption leads to fundamentally incorrect results; we
simply choose not to make it. We refer to \citet{arnold03} and \citet{duan15}
as general texts on random dynamical systems and stochastic differential
equations. \\

Based on these preliminaries, we consider a random differential equation
\begin{equation}
\dot{\mathbf{x}}=\mathbf{A}\mathbf{x}+\mathbf{f}_{0}\left(\mathbf{x}\right)+\epsilon\mathbf{g}\left(\mathbf{x},\boldsymbol{\theta}^{t}\left(\boldsymbol{\nu}\right)\right),\quad\mathbf{x}\in D\subset\mathbb{R}^{n},\quad0\leq\epsilon\ll1,\quad\mathbf{f}_{0}(\mathbf{x})=\mathcal{O}\left(\left|\mathbf{x}\right|^{2}\right),\label{eq:random ODE}
\end{equation}
on the probability space $\left(\mathcal{V},\mathcal{F},\mathbb{P}\right)$,
where $\mathbf{f}_{0}\colon D\subset\mathbb{R}^{n}\to\mathbb{R}^{n}$
is of smoothness class $C^{r}(D)$ with $r\geq2$. The uniformly bounded
function $\mathbf{g}$ is at least of class $C^{1}$ in \textbf{$\mathbf{x}\in D$},
$C^{0}$ in $t$ for fixed $\boldsymbol{\nu}$, and measurable in
$\boldsymbol{\nu}$. The domain $D$ is assumed compact and forward
invariant for $\epsilon=0$. An example of such a random ODE would
be a stochastic ODE that depends on a Brownian motion confined by
reflecting boundaries. Another practical example would be a non-autonomous
ODE whose time dependence is generated on the fly at discrete time
instances via a random variable with values in the $[0,1]$ interval.

\subsection{The unforced system ($\epsilon=0$)\label{subsec:The-unforced-system}}

We assume that $\mathbf{x}=0$ is a stable hyperbolic fixed point
for eq. (\ref{eq:random ODE}) and the set $D$ lies fully in the
domain of attraction of this fixed point. The linearized ODE at the
fixed point is then 
\begin{equation}
\dot{\mathbf{x}}=\mathbf{A}\mathbf{x},\label{eq: linearized system}
\end{equation}
and the eigenvalues $\left\{ \lambda_{j}\right\} _{j=1}^{n}$ of $\mathbf{A}$
can be ordered so that 
\begin{equation}
\mathrm{Re}\lambda_{n}\leq\mathrm{Re}\lambda_{n-1}\leq\ldots\leq\mathrm{Re}\lambda_{1}<0\label{eq:stable sectrum}
\end{equation}
holds. We assume, for simplicity, that $\mathbf{A}$ is diagonalizable,
in which case there are $n$ real eigenspaces $E_{1},$ $E_{2},\ldots,E_{n}$
corresponding to the eigenvalues listed in (\ref{eq:stable sectrum}).
Each eigenspace corresponding to a real eigenvalue is one- dimensional
and those corresponding to a complex conjugated pair are two-dimensional.
\\

We consider a \emph{slow spectral subspace} 
\begin{equation}
E=E_{1}\varoplus\ldots\varoplus E_{s},\quad\dim E=d,\label{eq:slow-subspace}
\end{equation}
spanned by the first $s$ eigenspaces of $\mathbf{A}$. If all the
eigenspaces correspond to single real eigenvalues, then $d=s$. If
all the eigenspaces correspond to single complex conjugate pairs,
then $d=2s$.

By construction, $E$ is a $d$-dimensional, attracting invariant
subspace of the linearized ODE (\ref{eq: linearized system}) that
is spanned by the $d$ slowest decaying solution families of this
ODE. After a possible linear change of coordinates, we may assume
that 
\begin{equation}
E=\left\{ \mathbf{x}\in\mathbb{R}^{n}\colon\,\mathbf{x}_{d+1}=\ldots=\mathbf{x}_{n}=0\right\} ,\quad\mathbf{A}=\left(\begin{array}{cc}
\mathbf{A}_{E} & 0\\
0 & \mathbf{B}
\end{array}\right),\label{eq:block_diagonal}
\end{equation}
with $\mathbf{A}_{E}\in\mathbb{R}^{d\times d}$ and $\mathbf{B}\in\mathbb{R}^{(n-d)\times(n-d)}$
. Restriction of the full linear ODE to $E$ gives an exact, $d$-dimensional
reduced-order model with which all solutions of (\ref{eq: linearized system})
synchronize once the fastest $n-d$ solution components of the system
have died out. Specifically, in appropriate coordinates 
\[
\mathbf{x}=\left(\boldsymbol{\xi},\boldsymbol{\eta}\right)\in\mathbb{R}^{d}\times\mathbb{R}^{\left(n-d\right)}
\]
that block-diagonalize $\mathbf{A}$, the $d$-dimensional reduced
linear dynamics on $E$ satisfies 
\[
\dot{\boldsymbol{\xi}}=\mathbf{A}_{E}\boldsymbol{\xi},\quad\mathrm{spect}\mathbf{A}_{E}=\left\{ \lambda_{1},\ldots,\lambda_{d}\right\} .
\]
An important question is whether such a model reduction procedure
could also be justified for the full nonlinear, random ODE (\ref{eq:random ODE}).

As a first step in answering this question, we define the \emph{spectral
quotient} of $E$ as the positive integer 
\[
\sigma\left(E\right)=\Biggl\lfloor\frac{\mathrm{Re}\lambda_{n}}{\mathrm{Re}\lambda_{1}}\Biggr\rfloor,
\]
with the half brackets referring to the integer part of a positive
real number. The slow spectral subspace $E$ is called \emph{nonresonant}
whenever the spectrum of the operator $\mathbf{A}_{E}$ has no low-order
resonance relationship with the spectrum of the operator $\mathbf{B}$.
More specifically, $E$ is nonresonant if 
\begin{equation}
\sum_{j=1}^{d}m_{j}\lambda_{j}\neq\lambda_{k},\quad2\leq\sum_{j=1}^{d}m_{j}\leq\sigma\left(E\right),\quad k>d.\label{eq:nonresonance}
\end{equation}

\citet{haller16}\emph{ }define a\emph{ spectral submanifold} $\mathcal{W}_{0}(E)$
of the deterministic ($\epsilon=0$) limit of system (\ref{eq:random ODE})
as the smoothest forward-invariant manifold that is tangent to the
spectral subspace $E$ at $x=0$ and has the same dimension as $E$.
Using results by \citet{cabre03}, \citet{haller16} deduce that such
a manifold is well-defined and of class $C^{r}$ as long as $r>\sigma\left(E\right)$
and the conditions (\ref{eq:stable sectrum}) and (\ref{eq:nonresonance})
are satisfied. In other words, $\mathcal{W}_{0}(E)$ exists and provides
a unique, class $C^{r}$, nonlinear continuation of the invariant
spectral subspace $E$ in the deterministic limit of the random ODE
(\ref{eq:random ODE}). Near the origin, $\mathcal{W}_{0}(E)$ can
therefore be written as a graph 
\begin{equation}
\boldsymbol{\eta}=\mathbf{h}_{0}(\boldsymbol{\xi}),\label{eq:h_0}
\end{equation}
for some function $\mathbf{h}_{0}\in C^{r}$. With the notation 
\[
\mathbf{f}_{0}(\mathbf{x})=\left(\mathbf{f}_{0\boldsymbol{\xi}}\left(\boldsymbol{\xi},\boldsymbol{\eta}\right),\mathbf{f}_{0\boldsymbol{\eta}}\left(\boldsymbol{\xi},\boldsymbol{\eta}\right)\right),
\]
the reduced dynamics on $\mathcal{W}_{0}(E)$ satisfies the reduced
ODE 
\begin{equation}
\dot{\boldsymbol{\xi}}=\mathbf{f}_{0\boldsymbol{\xi}}\left(\boldsymbol{\xi},\mathbf{h}_{0}(\boldsymbol{\xi})\right),\label{eq: deterministic reduced ODE}
\end{equation}
which is an exact reduced-order model for system (\ref{eq:random ODE})
for $\epsilon=0$ in a neighborhood of the origin. We show the geometry
of the subspace $E$ and the deterministic spectral submanifold $\mathcal{W}_{0}(E)$
in Fig. \ref{fig:Geometry of SSM}. 
\begin{figure}
\centering{}\includegraphics[width=0.5\textwidth]{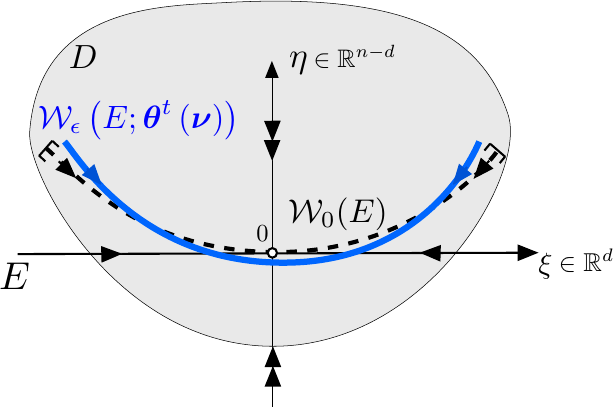}\caption{The geometry of the slow spectral subspace $E$, the deterministic
spectral submanifold $\mathcal{W}_{0}(E)$ and the random invariant
manifold $\mathcal{W}_{\epsilon}\left(E;\boldsymbol{\theta}^{t}\left(\boldsymbol{\nu}\right)\right)$.}
\label{fig:Geometry of SSM} 
\end{figure}

Of interest to us here is the continuation of $\mathcal{W}_{0}(E)$
and its reduced dynamics for $\epsilon>0$ in the full random ODE
(\ref{eq:random ODE}). To this end, we define the \emph{spectral
gap} $\rho\left(E\right)$ associated with $E$ as the integer part
of the ratio of the slowest decay rate outside $E$ to the fastest
decay rate inside $E$: 
\begin{equation}
\rho\left(E\right)=\Biggl\lfloor\frac{\mathrm{Re}\,\lambda_{d+1}}{\mathrm{Re}\,\lambda_{d}}\Biggr\rfloor.\label{eq:spectral-gap}
\end{equation}
Note that $1<\rho(E)\leq\sigma\left(E\right)$ by definition. If needed,
we reduce the size of the domain $D$ to make sure that all trajectories
in $\mathcal{W}(E)$ converge to the fixed point at the origin. In
that case, SSM-normal compression rates along each trajectory in $\mathcal{W}(E)$
will be at least $\rho\left(E\right)$-times stronger than the SSM-tangential
compression rates as $t\to\infty$. In other words, $W(E)$ is a $\rho\left(E\right)$-normally
hyperbolic invariant manifold by the classic definition of \citet{fenichel71}.
After possibly modifying the boundary of $\mathcal{W}(E)$, we can
also assume that $\mathcal{W}(E)$ is an inflowing invariant manifold,
i.e., all nonzero trajectories cross the boundary of $\mathcal{W}(E)$
strictly inwards.

\subsection{The forced system ($\epsilon>0$)}

With the help of the noise model $\boldsymbol{\theta}^{t}$ on the
probability space $\mathcal{V}$ (see eq. (\ref{eq:random ODE})),
the solution operator 
\begin{equation}
\mathbf{F}_{\epsilon}^{t}\colon X\times\mathcal{V}\to X,\qquad t\in\mathbb{R},\label{eq:randon flow map definiton}
\end{equation}
of the random ODE (\ref{eq:random ODE}) defines a \emph{random dynamical
system} (or \emph{random flow map)} $\mathbf{F}_{\epsilon}^{t}\colon D\to\mathbb{R}^{n}$
that is measurable in $t$ and over $\mathcal{V}\times X$, and satisfies
the \emph{cocycle condition} 
\begin{align}
\mathbf{F}_{\epsilon}^{0}(\mathbf{x};\boldsymbol{\nu}) & =\mathbf{x},\nonumber \\
\mathbf{F}_{\epsilon}^{t+s}(\mathbf{x};\boldsymbol{\nu}) & =\mathbf{F}_{\epsilon}^{s}\left(\mathbf{F}_{\epsilon}^{t}\left(\mathbf{x};\boldsymbol{\nu}\right);\boldsymbol{\theta}^{t}\left(\boldsymbol{\nu}\right)\right).\label{eq:cocycle condition}
\end{align}
In addition, such a random dynamical system is said to be \emph{class
$C^{r}$ smooth} if the mapping $\mathbf{F}_{\epsilon}^{t}(\,\cdot\,;\boldsymbol{\nu})$
is of class $C^{r}$ for all $\left(t,\boldsymbol{\nu}\right)$.

The second condition in (\ref{eq:cocycle condition}) ensures that
the evolution along random trajectories satisfies an analogue of the
group property of deterministic flows. Specifically, stitching together
subsequent segments of a full random trajectory should result in the
same final position for the piecewise stitched trajectory and for
the full trajectory, even though each segment of the stitched trajectory
is initialized under a different element of $\mathcal{V}$ chosen
at the start time of the segment. The metric dynamical system $\boldsymbol{\theta}^{t}$
ensures that these elements of $\mathcal{V}$ are properly (and deterministically)
related to each other along each random trajectory. In modeling the
noise in eq. (\ref{eq:random ODE}), the probability space $\mathcal{V}$
is commonly selected as the space of continuous paths $\left\{ \boldsymbol{\nu}(s)\right\} _{s\geq0}$
with $\boldsymbol{\nu}(0)=0.$ In that case, $\boldsymbol{\theta}^{t}\boldsymbol{\nu}\left(\,\cdot\,\right)=\boldsymbol{\nu}\left(t+\cdot\,\right)-\boldsymbol{\nu}(t)$
is called the Wiener shift under which the probability measure $\mathbb{P}$
is invariant and ergodic (see \citet{arnold03} or \citet{duan15}).

A \emph{random set $\mathcal{M}\left(\boldsymbol{\nu}\right)\subset D$
}is a family of nonempty closed sets such that 
\[
\boldsymbol{\nu}\mapsto\inf_{\mathbf{y}\in\mathcal{M}\left(\boldsymbol{\nu}\right)}\left\Vert \mathbf{x}-\mathbf{y}\right\Vert 
\]
is a random variable for any $\mathbf{x}\in D$. Such a random set
is a \emph{random manifold} if each $\mathcal{M}\left(\boldsymbol{\nu}\right)$
is a manifold. A random manifold is a \emph{random invariant manifold}
for the random flow map for system (\ref{eq:random ODE}) if 
\[
\mathbf{F}_{\epsilon}^{t}\left(\mathcal{M}\left(\boldsymbol{\nu}\right);\boldsymbol{\nu}\right)=\mathcal{M}\left(\boldsymbol{\theta}^{t}\left(\boldsymbol{\nu}\right)\right)
\]
holds for all $t\in\mathbb{R}$ and $\boldsymbol{\nu}\in\mathcal{V}$.
A random manifold $\bar{\mathcal{M}}\left(\boldsymbol{\nu}\right)=\mathcal{M}\left(\boldsymbol{\nu}\right)\cup\partial\mathcal{M}\left(\boldsymbol{\nu}\right)$
with boundary $\partial\mathcal{M}\left(\boldsymbol{\nu}\right)$
is called a \emph{random} \emph{inflowing invariant manifold} (with
boundary) if 
\[
\mathbf{F}_{\epsilon}^{t}\left(\mathcal{M}\left(\boldsymbol{\nu}\right);\boldsymbol{\nu}\right)\supset\bar{\mathcal{M}}\left(\boldsymbol{\theta}^{t}\left(\boldsymbol{\nu}\right)\right)
\]
for all $t<0$ and all $\boldsymbol{\nu}\in\mathcal{V}.$ Similarly,
a random manifold $\bar{\mathcal{M}}\left(\boldsymbol{\nu}\right)$
is a \emph{random} \emph{overflowing invariant manifold} (with boundary)
if 
\[
\mathbf{F}_{\epsilon}^{t}\left(\mathcal{M}\left(\boldsymbol{\nu}\right);\boldsymbol{\nu}\right)\supset\bar{\mathcal{M}}\left(\boldsymbol{\theta}^{t}\left(\boldsymbol{\nu}\right)\right)
\]
for all $t>0$ and all $\boldsymbol{\nu}\in\mathcal{V}.$

In Appendix A, we prove the following main theorem on the existence
of random SSMs by combining results of \citet{li13} and \citet{eldering18}
in our specific setting. 
\begin{thm}
\label{thm:main}Assume that 
\[
r\geq\rho\left(E\right).
\]
Then, for $\epsilon>0$ small enough:

(i) The random ODE (\ref{eq:random ODE}) has a class-$C^{\rho\left(E\right)}$,
random inflowing-invariant manifold $\mathcal{W}_{\epsilon}(E;\boldsymbol{\theta}^{t}\left(\boldsymbol{\nu}\right))$,
which can locally be written as 
\[
\mathcal{W}_{\epsilon}(E;\boldsymbol{\theta}^{t}\left(\boldsymbol{\nu}\right))=\left\{ \mathbf{x}=\left(\boldsymbol{\xi},\boldsymbol{\eta}\right)\in D\colon\,\boldsymbol{\eta}=\mathbf{h}_{\epsilon}\left(\boldsymbol{\xi};\boldsymbol{\theta}^{t}\left(\boldsymbol{\nu}\right)\right)=\mathbf{h}_{0}(\boldsymbol{\xi})+\epsilon\mathbf{h}_{1}\left(\boldsymbol{\xi};\boldsymbol{\theta}^{t}(\boldsymbol{\nu}),\epsilon\right)\right\} ,
\]
where $\mathbf{h}_{1}\left(\boldsymbol{\xi};\boldsymbol{\theta}^{t}(\boldsymbol{\nu}),\epsilon\right)$
is measurable in $\boldsymbol{\nu}$, and $C^{\rho\left(E\right)}$
smooth in $\boldsymbol{\xi}$ and $\epsilon$. Therefore, $\mathcal{W}_{\epsilon}\left(E;\boldsymbol{\theta}^{t}\left(\boldsymbol{\nu}\right)\right)$
is $\mathcal{O}\left(\epsilon\right)$ $C^{1}$-close to $\mathcal{W}(E)$
in $D.$

(ii) $\mathcal{W}_{\epsilon}\left(E;\boldsymbol{\theta}^{t}\left(\boldsymbol{\nu}\right)\right)$
is $C^{\rho\left(E\right)}$-diffeomorphic to $\mathcal{W}(E)$ inside
the domain $D$ for all $\boldsymbol{\nu}\in\mathcal{V}$.

(iii) $\mathcal{W}_{\epsilon}\left(E;\boldsymbol{\theta}^{t}\left(\boldsymbol{\nu}\right)\right)$
attracts any solutions of (\ref{eq:random ODE}) starting inside $D$
with probability $\mathbb{P}=1.$

(iv) $\mathcal{W}_{\epsilon}(E;\boldsymbol{\theta}^{t}(\boldsymbol{\nu}))$
can be locally written as 
\[
\mathcal{W}_{\epsilon}(E;\boldsymbol{\theta}^{t}(\boldsymbol{\nu}))=\left\{ \mathbf{x}=\left(\boldsymbol{\xi},\boldsymbol{\eta}\right)\in D\colon\,\boldsymbol{\eta}=\mathbf{h}_{\epsilon}(\boldsymbol{\xi};\boldsymbol{\theta}^{t}(\boldsymbol{\nu}))=\mathbf{h}_{0}(\boldsymbol{\xi})+\epsilon\mathbf{h}_{1}\left(\boldsymbol{\xi};\boldsymbol{\theta}^{t}(\boldsymbol{\nu}),\epsilon\right)\right\} .
\]

(v) The dynamics restricted to $\mathcal{W}_{\epsilon}(E;\boldsymbol{\theta}^{t}(\boldsymbol{\nu}))$
satisfies the reduced random ODE 
\begin{align}
\dot{\boldsymbol{\xi}} & =\mathbf{f}_{0\boldsymbol{\xi}}\left(\boldsymbol{\xi},\mathbf{h}_{\epsilon}\left(\boldsymbol{\xi};\boldsymbol{\theta}^{t}\left(\boldsymbol{\nu}\right)\right)\right)+\epsilon\mathbf{g}_{\boldsymbol{\xi}}\left(\boldsymbol{\xi},\mathbf{h}_{\epsilon}\left(\boldsymbol{\xi};\boldsymbol{\theta}^{t}\left(\boldsymbol{\nu}\right)\right),\boldsymbol{\theta}^{t}\left(\boldsymbol{\nu}\right)\right)\nonumber \\
 & =\mathbf{f}_{0\boldsymbol{\xi}}\left(\boldsymbol{\xi},\mathbf{h}_{0}\left(\boldsymbol{\xi}\right)\right)+\epsilon\left[D_{\boldsymbol{\eta}}\mathbf{f}_{0\boldsymbol{\xi}}\left(\boldsymbol{\xi},\mathbf{h}_{0}\left(\boldsymbol{\xi}\right)\right)+\mathbf{g}_{\boldsymbol{\xi}}\left(\boldsymbol{\xi},\mathbf{h}_{0}\left(\boldsymbol{\xi}\right),\boldsymbol{\theta}^{t}\left(\boldsymbol{\nu}\right)\right)\right]+\mathcal{O}\left(\epsilon^{2}\right).\label{eq:random reduced ODE}
\end{align}
\end{thm}
\begin{proof}
See Appendix A. 
\end{proof}
We show schematically the random inflowing-invariant manifold $\mathcal{W}_{\epsilon}(E;\boldsymbol{{\theta}}^{t}(\boldsymbol{\nu}))$
for one particular $\boldsymbol{\nu}\in\mathcal{V}$ in Fig. \ref{fig:Geometry of SSM}.
By definition, as any random invariant manifold, $\mathcal{W}_{\epsilon}(E;\boldsymbol{{\theta}}^{t}(\boldsymbol{\nu}))$
is time-dependent in any realization of the random ODE (\ref{eq:random ODE})
and satisfies 
\[
\mathbf{F}_{\epsilon}^{t}\left(\mathcal{W}_{\epsilon}(E;\boldsymbol{\nu});\boldsymbol{\nu}\right)\supset\mathcal{W}_{\epsilon}\left(E;\boldsymbol{\theta}^{t}\left(\boldsymbol{\nu}\right)\right)
\]
for all $t\geq0$.

By statement (iv) of Theorem \ref{thm:main}, model reduction to the
persisting random invariant manifold $\mathcal{W}_{\epsilon}\left(E;\boldsymbol{\theta}^{t}\left(\boldsymbol{\nu}\right)\right)$
proceeds along the same lines as in the deterministic case, allowing
one to reduce high-dimensional Monte-Carlo simulations to low-dimensional
ones. Note that only $\rho\left(E\right)$ continuous derivatives
can be guaranteed for the random SSM even though its deterministic
limit $\mathcal{W}_{0}(E)$ is of the generally higher smoothness
class $C^{r}$. 
\begin{rem}
\textbf{{[}Computing the SSM-reduced model{]}} The graph (\ref{eq:h_0})
of the unforced, deterministic SSM and can be calculated to a high
degree of accuracy using existing the MATLAB software \textit{SSMTool}
(see \citet{jain2022} and \citet{jain23}). Therefore, the leading-order
reduced-order model (\ref{eq:random reduced ODE}) on the random SSM
$\mathcal{W}_{\epsilon}\left(E;\boldsymbol{\theta}^{t}\left(\boldsymbol{\nu}\right)\right)$
can be simulated by using $\mathbf{h}_{0}$ obtained from SSMTool.
Note that the block-diagonalization used in eq. (\ref{eq:block_diagonal})
for mathematical exposistion is not carried out in SSMTool. Rather,
the computations are done directly in the coordinates $\mathbf{{x}}$
appearing in the random ODE (\ref{eq:random ODE}).
\end{rem}
\begin{rem}
\textbf{{[}State-independent random forcing{]}} In applications, the
random forcing model frequently has no dependence on the phase space
variable $\mathbf{x}$. In that case, we have $\mathbf{g}_{\boldsymbol{\xi}}\left(\boldsymbol{\xi},\mathbf{h}_{0}\left(\boldsymbol{\xi}\right),\boldsymbol{\theta}^{t}\left(\boldsymbol{\nu}\right)\right)\equiv\mathbf{g}_{\boldsymbol{\xi}}\left(\boldsymbol{\theta}^{t}\left(\boldsymbol{\nu}\right)\right)$
and hence $\mathbf{D}_{\boldsymbol{\eta}}\mathbf{f}_{0\boldsymbol{\xi}}\equiv\mathbf{0}$
in the reduced equation (\ref{eq:random reduced ODE}). Therefore,
we can simply add the projection of the full random forcing on the
spectral subspace $E$ to the unforced reduced dynamics in $\mathcal{W}_{0}\left(E\right)$
to obtain the reduced random dynamical system
\begin{align}
\dot{\boldsymbol{\xi}} & =\mathbf{f}_{0\boldsymbol{\xi}}\left(\boldsymbol{\xi},\mathbf{h}_{0}\left(\boldsymbol{\xi}\right)\right)+\epsilon\mathbf{g}_{\boldsymbol{\xi}}\left(\boldsymbol{\theta}^{t}\left(\boldsymbol{\nu}\right)\right)+\mathcal{O}\left(\epsilon^{2}\right).\label{eq:random reduced ODE-1-1}
\end{align}
\end{rem}
\begin{rem}
\textbf{{[}Relation to the case of non-random, aperiodic forcing{]}}
Under a single realization of the uniformly bounded random forcing,
the forcing becomes deterministic but still aperiodic. \citet{haller24}
gives asymptotic formulas for the corresponding deterministic, perturbed
SSM $\mathcal{W}_{\epsilon}(E,t)$ up to any finite order of accuracy. 
\end{rem}
\begin{rem}
\textbf{{[}Smoothness of the random SSM{]}} For simplicity, we have
considered the domain $D$ small enough so that $\mathcal{W}_{0}(E)$
had very simple dynamics: all solutions in the deterministic SSM converged
to the origin. This enables us to establish that the strength of normal
hyperbolicity (i.e., how many times normal attraction overpowers tangental
compression along $\mathcal{W}_{0}(E)$) is precisely the spectral
gap $\rho\left(E\right)$ defined in eq. (\ref{eq:spectral-gap}),
which can be computed purely from the spectrum of the matrix $\mathbf{A}$.
This, in turn, enables us to conclude that the random invariant manifold
$\mathcal{W}_{\epsilon}\left(E;\boldsymbol{\theta}^{t}\left(\boldsymbol{\nu}\right)\right)$
is $C^{\rho\left(E\right)}$ smooth (see the Appendix for details).
On larger $D$ domains, however, $\mathcal{W}_{0}(E)$ may have nontrivial
internal dynamics, i.e., may contain further limit sets in addition
to the fixed point at $\mathbf{x}=0$. In that case, an analog of
Theorem \ref{thm:main} continues to hold as long as $\mathcal{W}_{0}(E)$
remains normally hyperbolic, i.e., normal attraction rates to $\mathcal{W}_{0}(E)$
still overpower tangential compression rates everywhere along $\mathcal{W}_{0}(E)$
as $t\to\infty.$ Establishing the strength of normal hyperbolicity
of $\mathcal{W}_{0}(E)$ in that case, however, requires the exact
knowledge of the asymptotic ratio of normal and tangential compression
rates along all limit sets within $\mathcal{W}(E)$. This strength
(and hence the guaranteed smoothness) of the persisting $\mathcal{W}_{\epsilon}\left(E;\boldsymbol{\theta}^{t}\left(\boldsymbol{\nu}\right)\right)$
will then be the integer part of the minimum of these ratios taken
over all limit sets of $\mathcal{W}_{0}(E)$. 
\end{rem}
\begin{rem}
\textbf{{[}Existence of random SSMs for repelling fixed points{]}}
Results similar to those listed in Theorem \ref{thm:main} hold when
the origin is a repelling fixed point, i.e., the matrix $A$ has eigenvalues
\begin{equation}
\mathrm{Re}\lambda_{n}\geq\mathrm{Re}\lambda_{n-1}\geq\ldots\geq\mathrm{Re}\lambda_{1}>0.\label{eq:stable sectrum-1}
\end{equation}
In that case, the surviving random invariant manifold $\mathcal{M}_{\epsilon}(\boldsymbol{}{\theta}^{t}(\boldsymbol{\nu}))$
is overflowing and repels all initial conditions in $D$ for small
enough $\epsilon>0$. 
\end{rem}

\section{Examples}

We consider mechanical systems of the form 
\begin{equation}
\mathbf{M}\ddot{\mathbf{q}}+\mathbf{C\dot{q}}+\mathbf{Kq}+\mathbf{f}_{nl}(\mathbf{q},\mathbf{\dot{q}})=\mathbf{p}\left(\boldsymbol{\theta}^{t}\left(\boldsymbol{\nu}\right),\mathbf{{q},{\dot{{q}}}}\right),\label{eq:eom}
\end{equation}
where $\mathbf{M,C,K}\in\mathbb{R}^{n\times n}$ are positive positive
definite mass matrix, and the positive semidefinite damping and stiffness
matrices, respectively; $\mathbf{q\in}\mathbb{R}^{n}$ is the vector
of generalized coordinates; the function $\mathbf{f}_{nl}(\mathbf{q},\mathbf{\dot{q}})$
is a vector of geometric and material nonlinearities; and $\mathbf{p}\left(\boldsymbol{\theta}^{t}\left(\boldsymbol{\nu}\right),\mathbf{{q},{\dot{{q}}}}\right)$
is a uniformly bounded external or parametric forcing vector with
random time dependence. 

The second order system (\ref{eq:eom}) can be rewritten in its first
order system form

\begin{equation}
\dot{\mathbf{x}}=\mathbf{Ax}+\mathbf{f}_{0}(\mathbf{x})+\mathbf{g}(\boldsymbol{\theta}^{t}\left(\boldsymbol{\nu}\right),\mathbf{{x}}),\quad\mathbf{x}=\left(\mathbf{q},\dot{\mathbf{q}}\right)^{\mathrm{T}},\label{eq:first order system}
\end{equation}
where 
\[
\mathbf{A=\left[\begin{array}{cc}
\mathbf{0} & \mathbf{\mathbf{I}}\\
-\mathbf{M}^{-1}\mathbf{K} & -\mathbf{M}^{-1}\mathbf{C}
\end{array}\right]},\quad\mathbf{f}_{0}(\mathbf{x})=\left[\begin{array}{c}
\mathbf{0}\\
-\mathbf{M}^{-1}\mathbf{f}_{nl}(\mathbf{q},\mathbf{\dot{q}})
\end{array}\right],
\]

\[
\text{ and \quad}\mathbf{g}(\boldsymbol{\theta}^{t}\left(\boldsymbol{\nu}\right),\mathbf{{x}})=\left[\begin{array}{c}
\mathbf{0}\\
\mathbf{M}^{-1}\mathbf{p}\left(\boldsymbol{\theta}^{t}\left(\boldsymbol{\nu}\right),\mathbf{{q},{\dot{{q}}}}\right)
\end{array}\right].
\]
In all our examples, the unforced system admits an asymptotically
stable fixed point at $\mathbf{x}=0$. We will assume that a uniform
bound on the forcing is defined as 

\[
\epsilon=\limsup_{t\in\mathbb{R},\boldsymbol{\nu}\in\mathcal{V},\mathbf{{x}\in}D}\left|\mathbf{g}(\boldsymbol{\theta}^{t}\left(\boldsymbol{\nu}\right),\mathbf{{x}})\right|
\]
for $D\subset\mathbb{{R}}^{2n}$, i.e., we subsume the book-keeping
parameter $\epsilon$ appearing in eq. (\ref{eq:random ODE}) into
the definition of the forcing for convenience.

\subsection*{Random force generation}

In our examples, the forcing vector $\mathbf{p}\left(\boldsymbol{\theta}^{t}\left(\boldsymbol{\nu}\right),\mathbf{{q},{\dot{{q}}}}\right),$
will have the form $\epsilon\theta^{t}(\nu)\mathbf{\,\mathbf{({v}+}{f}_{p}(\mathbf{q,{\dot{{q}}}})})$
where $\mathbf{\mathbf{{v}}}\in\mathbb{R}^{n}$ is a constant vector,
$\mathbf{\mathbf{{f_{p}}}}\in\mathbb{R}^{n}$ is a function representing
parametric excition, $\epsilon$ is a parameter controlling the magnitude
of the uniformly bounded forcing, $\theta^{t}(\nu)$ is a time-dependent
scalar random variable appearing in eq. (\ref{eq:random ODE}) that
is uniformly bounded in time for any $\nu$. We list three possible
ways to generate such noise: 
\begin{itemize}
\item \textbf{Method 1: Force generation from a spectral density.} In some
examples the random forcing signal $\theta^{t}(\nu)$ can be inferred
from an empirical spectral density $\Phi_{\nu}(\omega)=\mathbb{E}[(\theta^{\omega}(\nu))^{2}]$.
Here $\mathbb{E}$ is the expectation value operator and $\theta^{\omega}(\nu)$
is the Fourier transform of one realization of the random forcing
signal $\theta^{t}(\nu)$. Following textbook methods (see \citet{preumont2013random}),
we can numerically evaluate a realization of a forcing signal $\theta^{t}(\nu)$
from the spectral density $\Phi_{\nu}(\omega)$. We specify a frequency
range $\omega\in[\omega_{0},\omega_{M}]$ which is equally spaced
with spacing $\Delta\omega$. The forcing signal for one realization
in this frequency window can then be written as 
\begin{equation}
\theta^{t}(\nu)=\sum_{i=0}^{M}\sqrt{2\Phi_{\nu}(\omega_{0}+i\Delta\omega)\Delta\omega}\cos((\omega_{0}+i\Delta\omega)t+\phi_{i}),\label{eq:quasi-random}
\end{equation}
where $\omega_{N}=\omega_{0}+M\Delta\omega$ and $\phi_{i}$ is a
uniform random number from the interval $[0,2\pi]$. From a numerical
perspective, eq. (\ref{eq:quasi-random}) represents a fast Fourier
transform of the spectral density. Note that for every realization,
one can regard this signal to lie on a $N$-- dimensional torus,
thus the process is only quasi-random. \label{p:1}
\item \textbf{Method 2: L\'evy process with a bounded random number generator.}
A classic method to generate a truly random process in time is to
advect a random ODE describing Brownian motion. However in that case,
the random process is an Itô process and the noise generation is sampled
from a Gaussian distribution $\mathcal{{N}}$ which is unbounded.
This translates to the increments of the random variable given by
$(\theta^{t+\Delta t}(\nu)-\theta^{t}(\nu))$ following the probability
density function $\mathcal{N}(0,\Delta t)$ for all $t\geq0$. We
make this process bounded by prescribing the increments to follow
a truncated Gaussian distribution $\mathcal{NT}(0,\Delta t;a,b)$,
where the density function is $0$ outside the amplitude interval
$[a,b]$. If we choose a symmetric interval and also ensure the interval
encompasses regions of high probability, the truncated Gaussian increments
offers a bounded Brownian motion which is very close to the one generated
by an It\^o process. Since all our examples are mechanical systems,
it is convenient to model the randomness using a second-order system
(see \citet{KOZIN198858}) of the form 
\begin{equation}
m\ddot{a}+c\dot{a}+ka=\theta^{t}(\nu),\label{eq:filter}
\end{equation}
where $\ddot{a}$,$\dot{a}$ or $a$ are interpreted as physically
relevant scalar quantities like ground acceleration during an earthquake,
uniform fluid velocity or road elevation. The forcing $\theta^{t}(\nu)$
is a random process generated from a truncated Gaussian distribution.
This setup offers practical benefits when implemented in a numerical
scheme as one can evaluate the random forced response in real time
by coupling the above ODE to the mechanical system (\ref{eq:eom}).
\label{p:2}
\item \textbf{Method 3: It\^o process with reflective boundary conditions.}
Another possible method to implement a bounded noise generator is
to impose doubly reflective boundary conditions on one of the stochastic
random variables generated from eq. (\ref{eq:filter}) using a Gaussian
noise model for $\theta^{t}(\nu)$. The generated random variable
will always be confined to $[-1,1]$. \label{p:3}
\end{itemize}

\subsection*{Random forced response evaluation}

In our examples, the linear part of the unforced mechanical system
has a slow spectral subspace $E$ with $d=2$ (see eq. (\ref{eq:slow-subspace})).
Therefore, under uniformly bounded random forcing, Theorem \ref{thm:main}
is applicable. This allows us to construct a 2D time-dependent random
SSM $\mathcal{W}_{\epsilon}\left(E;\boldsymbol{\theta}^{t}\left(\boldsymbol{\nu}\right)\right)$.
Specifically, we use \textit{SSMtool} by \citet{jain23} to compute
the autonomous SSM $\mathcal{{W}}_{0}(E)$ as a graph $\boldsymbol{\eta}=\mathbf{h}_{0}(\boldsymbol{\xi})$
and its reduced dynamics $\mathbf{f}_{\boldsymbol{\xi}}(\boldsymbol{\xi},\mathbf{h}_{0}(\boldsymbol{\xi}))$
as order-$N$ multivariate polynomials. We then explicitly compute
the leading-order corrections (listed in Theorem \ref{thm:main})
to obtain the random SSM-reduced model.

We test the the accuracy of the random reduced-order models obtained
in this fashion by performing Monte Carlo simulations. We repeatedly
sample $m$ random forcing realizations using one of the outlined
methods (see Methods 1-3) and record the random forced response outcomes
of the full and reduced-order models for every realization. We further
compute the averaged spectral density (PSD) matrix $\boldsymbol{\Phi}_{\mathbf{x}}(\omega)\in\mathbb{C}^{2n\times2n}$
defined as (see \citet{papoulis1965random}) 
\begin{equation}
\boldsymbol{\Phi}_{\mathbf{x}}(\omega)=\mathbb{E}[\mathbf{x}(\omega)(\mathbf{x}(\omega))^{\dagger}]=\frac{1}{m}\sum_{j=1}^{m}\mathbf{x}^{j}(\omega)(\mathbf{x}^{j}(\omega))^{\dagger},
\end{equation}
where $\mathbf{x}^{j}(\omega)$ is the fast-Fourier transform of the
forced outcome $\mathbf{x}^{j}(t)$ and $[\cdot]^{\dagger}$ denotes
the complex conjugate transpose operation. Note that the diagonal
entries of $\boldsymbol{\Phi}_{\mathbf{x}}(\omega)$ are real and
represent the energy content of our outcomes. We will plot one of
these diagonal entries against the frequency $\omega$ to visualize
the performance of the reduced-order model in comparison to the full
system. 

We label a diagonal entry of the matrix as $\phi_{\mathbf{x}}^{f}(\omega)$,
where $f$ refers to a specific position or velocity component of
the system. For time integration, we use implicit methods for the
full system (see \citet{geradin2015mechanical}). For the full system
simulation in Method 1, we use the implicit Newmark scheme. In Methods
2 and 3, we formulate a second-order implicit scheme provided in \citet{secondOrderSDE}.
For the time integration of SSM-reduced systems, we apply a Runge-Kutta
scheme as described in \citet{SDE}. We also advect both the full
and the SSM-reduced models from the same initial condition, which
always lies on the random SSM.

\subsection{Suspension system moving on an irregular road\label{subsec:Suspension-system-moving}}

\label{sec:Suspension}

For the design of the suspension of a car or an offshore platform,
one needs to model the stability of these systems under road or sea
irregularity. We use here a simplified model of such a suspension
proposed by \citet{quarter-car}, as shown in Fig. \ref{fig:quater_car}.
We modify the original example by adding nonlinear cubic springs that
also introduces parametric forcing.

The equations of motion (\ref{eq:eom}), with $\mathbf{q}=(x_{s},x_{u})^{\top}$
denoting the displacements of the spring blocks in the suspension
system, are specifically
\begin{align}
\mathbf{M}=\begin{bmatrix}m_{s} & 0\\
0 & m_{u}
\end{bmatrix},\quad\mathbf{C}=\begin{bmatrix}c_{s} & -c_{s}\\
-c_{s} & c_{s}+c_{u}
\end{bmatrix},\quad\mathbf{K}=\begin{bmatrix}k_{s} & -k_{s}\\
-k_{s} & k_{s}+k_{u}
\end{bmatrix},\label{eq:suspension_system}
\end{align}
\[
\text{ and \ensuremath{\quad}}\mathbf{f}_{nl}(\mathbf{q})=\begin{pmatrix}\kappa_{1}(x_{s}-x_{u})^{3}\\
-\kappa_{1}(x_{s}-x_{u})^{3}-\kappa_{2}x_{u}^{3}
\end{pmatrix}.
\]

\begin{figure}[H]
\begin{centering}
\subfloat[]{\begin{centering}
\includegraphics[width=0.45\textwidth]{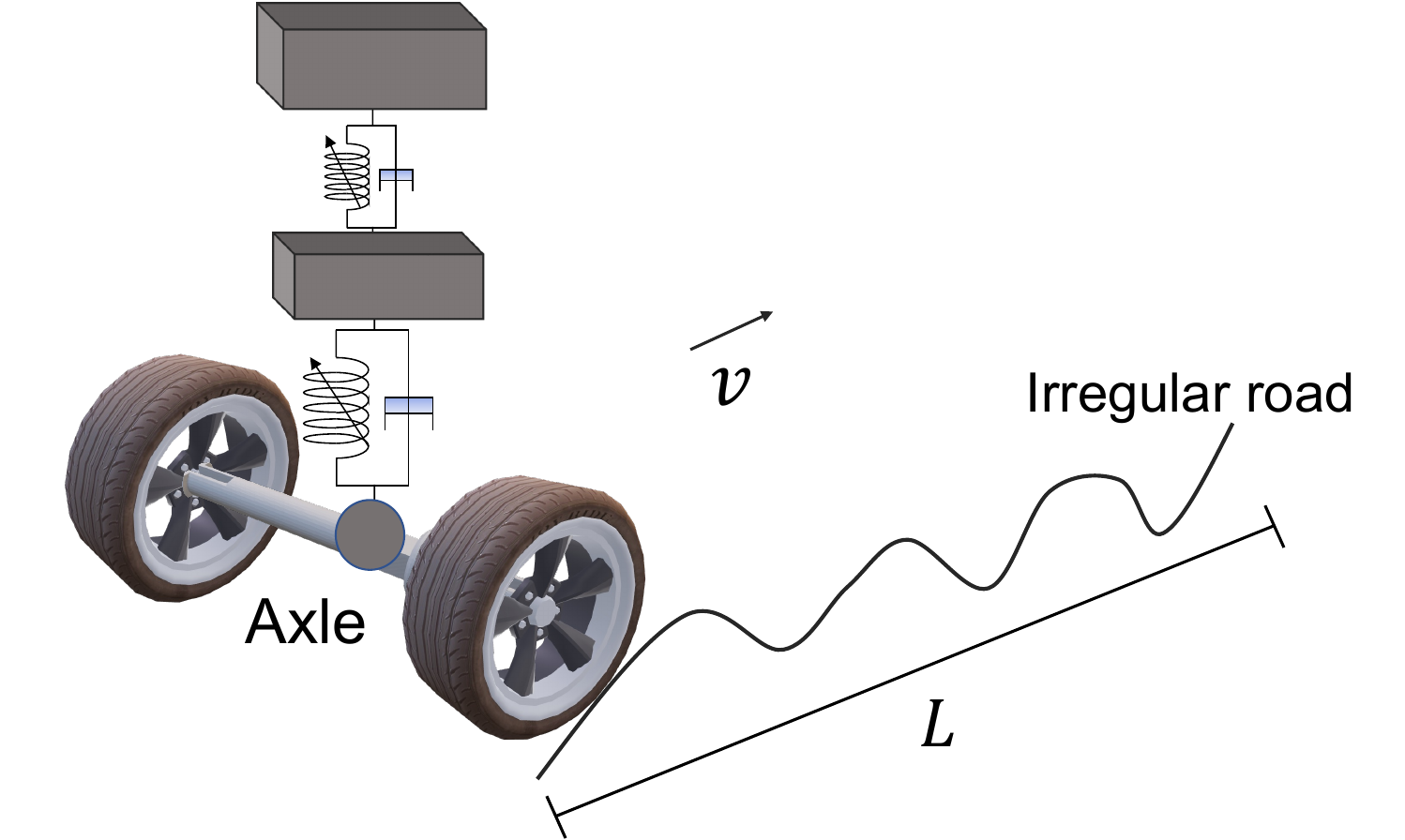} 
\par\end{centering}
}\qquad{}\qquad{}\subfloat[]{\begin{centering}
\includegraphics[width=0.35\textwidth]{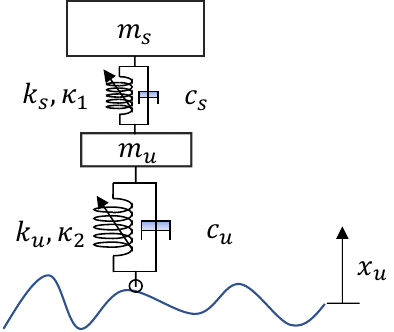} 
\par\end{centering}
}
\par\end{centering}
\caption{Left: An illustration of the suspension model of \citet{quarter-car}
moving with speed $v$ along an irregular road of length $L$ . Right:
A side profile of the model with system parameters shown.}
\label{fig:quater_car} 
\end{figure}

In eq. (\ref{eq:suspension_system}), we have ignored gravity for
simplicity, which will slightly perturb the equilibrium configuration
from $\mathbf{x}=\mathbf{{0}}$. We further assume that the system
is travelling at a constant speed $v$ along a road of length $L.$
This allows us to write the force exerted by the irregular road on
the system as 
\begin{equation}
\mathbf{p}\left(\boldsymbol{\theta}^{t}\left(\boldsymbol{\nu}\right),{\mathbf{{q}}},{\mathbf{{\dot{{q}}}}}\right)=\epsilon\begin{pmatrix}0\\
k_{u}\theta^{t}(h)+c_{u}\theta^{t}(\nabla h)v+\kappa_{2}[\theta^{t}(h)]^{3}-3\kappa_{2}[\theta^{t}(h)]^{2}x_{u}+3\kappa_{2}x_{u}^{2}\theta^{t}(h)
\end{pmatrix},
\end{equation}
where $h$ and $\nabla h$ are random variables denoting road elevation
and gradient, and $\epsilon$ is a measure of the overall magnitude
of forcing. There are empirical expressions for the spectral densities
to model road elevation (see \citet{quarter-car}) given by 
\begin{equation}
\phi_{h}(\omega)=\frac{A_{v}bv}{(bv)^{2}+\omega^{2}},
\end{equation}
where the coefficients $A_{v}$ and $b$ depend on the road irregularity,
and $v$ is the speed of the system as mentioned.

We estimate the spectral density of the gradient from the elevation
as $\phi_{\nabla h}(\omega)=\frac{\omega^{2}}{v}\phi_{h}(\omega)$.
We then use Method 1 to sample for the stochastic signals $\theta^{t}(h)$
and $\theta^{t}(\nabla h)$ until the car completes the travel time
$\frac{L}{v}$ on the road. We perform the sampling for $50$ forcing
realizations ($m=50$). We set the system parameters as $m_{s}=229\text{ [kg]}$,
$m_{t}=31\text{ [kg]}$, $c_{s}=120\text{ [Ns/m]}$, $c_{u}=120\text{ [Ns/m]}$,
$k_{s}=60\text{ [kN/m]}$, $k_{u}=20\text{ [kN/m]}$, $\kappa_{1}=250\text{ [kN/m\ensuremath{^{3}}]}$,
$\kappa_{2}=30\text{ [kN/m\ensuremath{^{3}}]}$ and $v=30\text{ [m/s]}$.
We set the forcing parameters to be $L=1800\text{ [m]}$, $A_{v}=3.5\times10^{-5}\text{ [m\ensuremath{^{2}}]}$
and $b=0.4\text{ [rad/m]}$. The values of road roughness are taken
from \citet{quarter-car}.

\begin{figure}[H]
\begin{centering}
\includegraphics[width=1\textwidth]{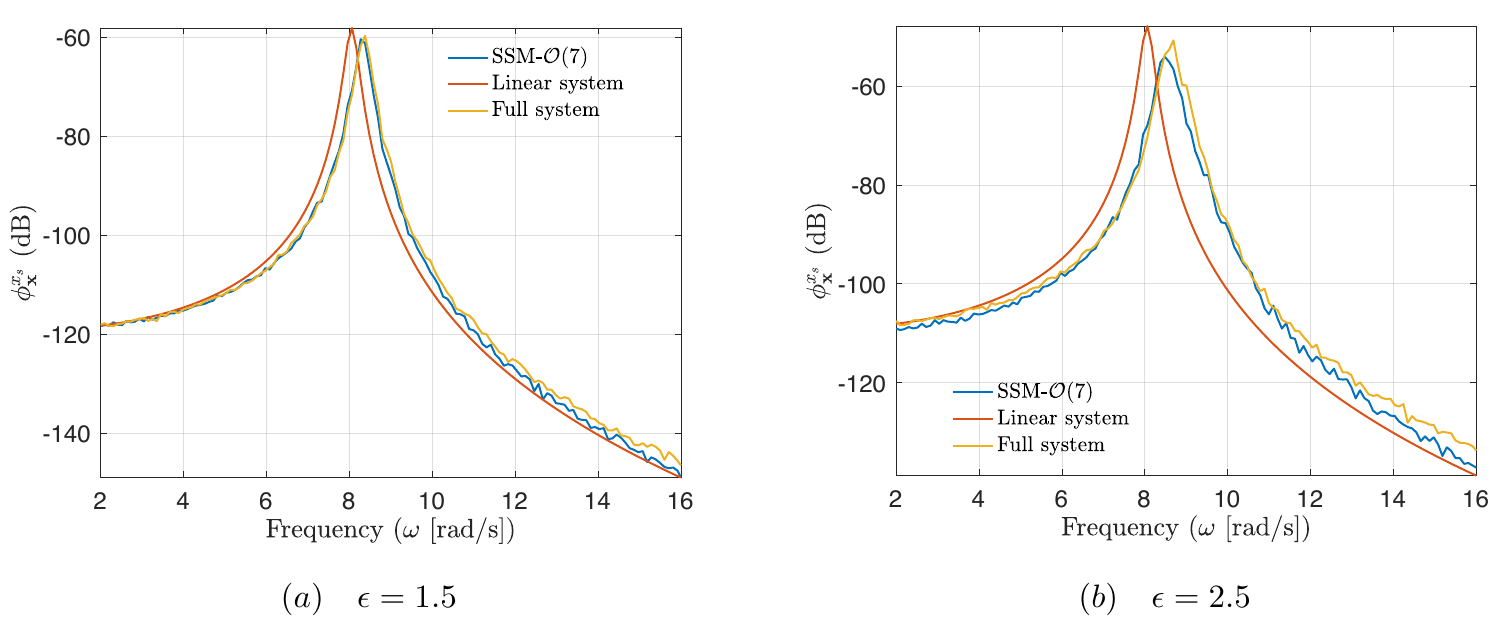} 
\par\end{centering}
\caption{ PSD $\phi_{\mathbf{x}}^{x_{s}}$ of the displacement of mass $m_{s}$
obtained from the SSM-reduced order model (blue) and for the full
order model simulation (yellow) for two regimes of road irregularities.
We also plot in red the analytic PSD for the linear system. (a) $\epsilon=1.5$
indicates minor road elevations (b) $\epsilon=2.5$ the road irregularities
are amplified. }
\label{fig:psd_qc} 
\end{figure}

Figure \ref{fig:psd_qc} shows the power spectral density $\phi_{\mathbf{x}}^{x_{s}}(\omega)$
of the displacement of the upper mass $m_{s}$ in decibel {[}dB{]},
calculated for the linear system, the full nonlinear system and the
random SSM-reduced model with an order $N=7$ multivariate Taylor
expansion for two regimes of road irregularity. When the road irregularity
is minimal ($\epsilon=1.5$), the linear response is close to the
the nonlinear response of the system, and our random SSM-reduced model
reproduces this trend. As expected, once we increase the road irregularity,
the nonlinear system deviates from the linear one and our leading-order
SSM-reduced model captures this trend. The linear predictions are
smooth because we can compute the response statistics analytically
(see Appendix B for details).

\subsection{Earthquake response of a building}

\label{sec:building}

Inspired by building models in civil engineering for modelling earthquake
response (see \citet{seismicModel}), we now consider a simple $n$-storey
building modelled by a vertical oscillator chain as shown in Fig.
\ref{fig:building}. Each mass represents a floor and gravity is neglected
for simplicity. The system is of the form (\ref{eq:eom}) with 
\begin{align}
\mathbf{q} & =[u_{1}\,,u_{2},\,...\,u_{n}]^{\top},\quad\mathbf{C}=\beta\mathbf{K}+\alpha\mathbf{M},\nonumber \\
\mathbf{M} & =\begin{bmatrix}m_{1} & 0 & 0\\
0 & \ddots & 0\\
0 & 0 & m_{n}
\end{bmatrix}\text{\ensuremath{\quad}{and}}\quad\mathbf{K}=\begin{bmatrix}k_{1}+k_{2} & -k_{2} & 0 & \cdots & \cdots\\
-k_{2} & k_{2}+k_{3} & -k_{3} & 0 & \cdots\\
0 & -k_{3} & k_{3}+k_{4} & \cdots & \cdots\\
\vdots & \vdots & \vdots & \ddots & 0\\
\vdots & \vdots & \vdots & k_{n-1}+k_{n} & -k_{n-1}\\
0 & \cdots & 0 & -k_{n-1} & k_{n}
\end{bmatrix}.
\end{align}
The nonlinearity $\mathbf{f}_{nl}(\mathbf{q},\mathbf{\dot{q}})$ in
this example comprises cubic nonlinear springs placed between neighbouring
storeys with hardening parameter $\kappa$. To model earthquake response,
we use Method 3 to generate bounded horizontal ground acceleration
noise. Specifically, we use the random variable ${a}$ in eq. (\ref{eq:filter})
with parameters $k=2\times10^{-2}\text{ [N/m]}$, $c=1\times10^{-1}\text{ [Ns/m]}$
and $m=5\times10^{-3}\text{ [kg]}$. The forcing vector is $\mathbf{p}\left(\boldsymbol{\theta}^{t}\left(\boldsymbol{\nu}\right)\right)=\epsilon\ddot{u}_{G}(m_{1},\dots,m_{n})^{\top}$,
where $\epsilon$ is a measure of the ground acceleration intensity
and $\ddot{u}_{G}=a$.

\begin{figure}[H]
\begin{centering}
\includegraphics[width=0.6\textwidth]{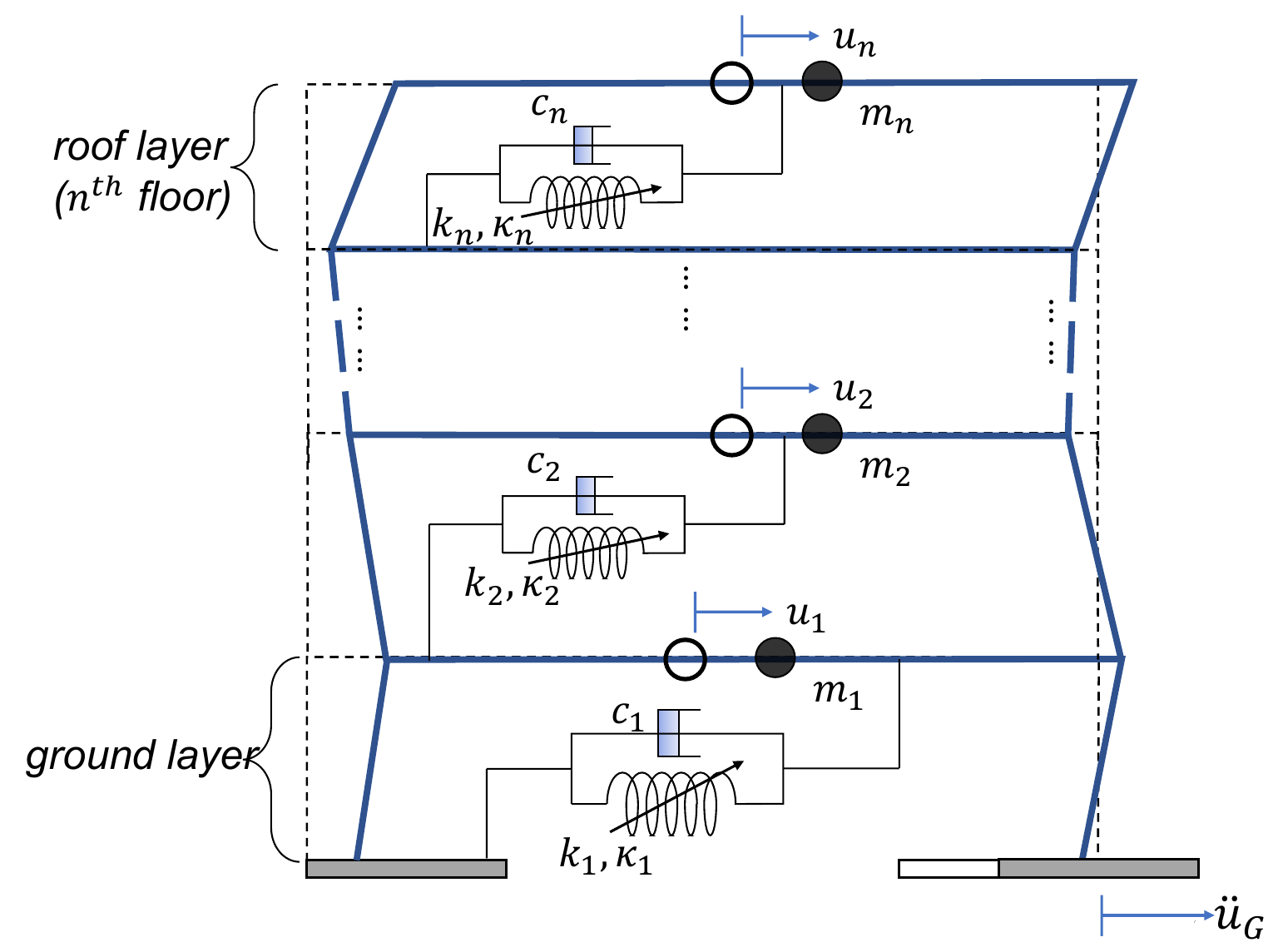} 
\par\end{centering}
\caption{Schematic model of $n$-storey building under seismic excitation.}
\label{fig:building} 
\end{figure}

We set the building parameters as $m_{1}=m_{2}=...=m_{n}=7\text{ [kg]}$,
$\alpha=0\text{ [Hz]}$, $\beta=0.0198\text{ [s]}$, $k_{1}=...=k_{n}=4555\text{ [N/m]}$
and $\kappa=2000\text{ [N/m\ensuremath{^{3}}]}$. The total number
of floors in our building is $n=10$. Similarly to the previous example,
we perform $m=50$ Monte Carlo simulations and compute the averaged
PSD for the degree of freedom $u_{10}$, which represents the roof
of our building model. We find that computing an order $N=5$ random
SSM-reduced model is sufficient to reproduce the statistics of the
full system.

In contrast to the suspension model in Section \ref{subsec:Suspension-system-moving},
the current setup induces truly stochastic forcing in time (see Methods
1-3); is maximally coupled with nonlinear springs; and is of higher
dimensions. Our plots in Figs. \ref{fig:psd_building}a, b indicate
that a 2D random SSM model suffices to capture the PSD of the full
system.

\begin{figure}[H]
\begin{centering}
\includegraphics[width=1\textwidth]{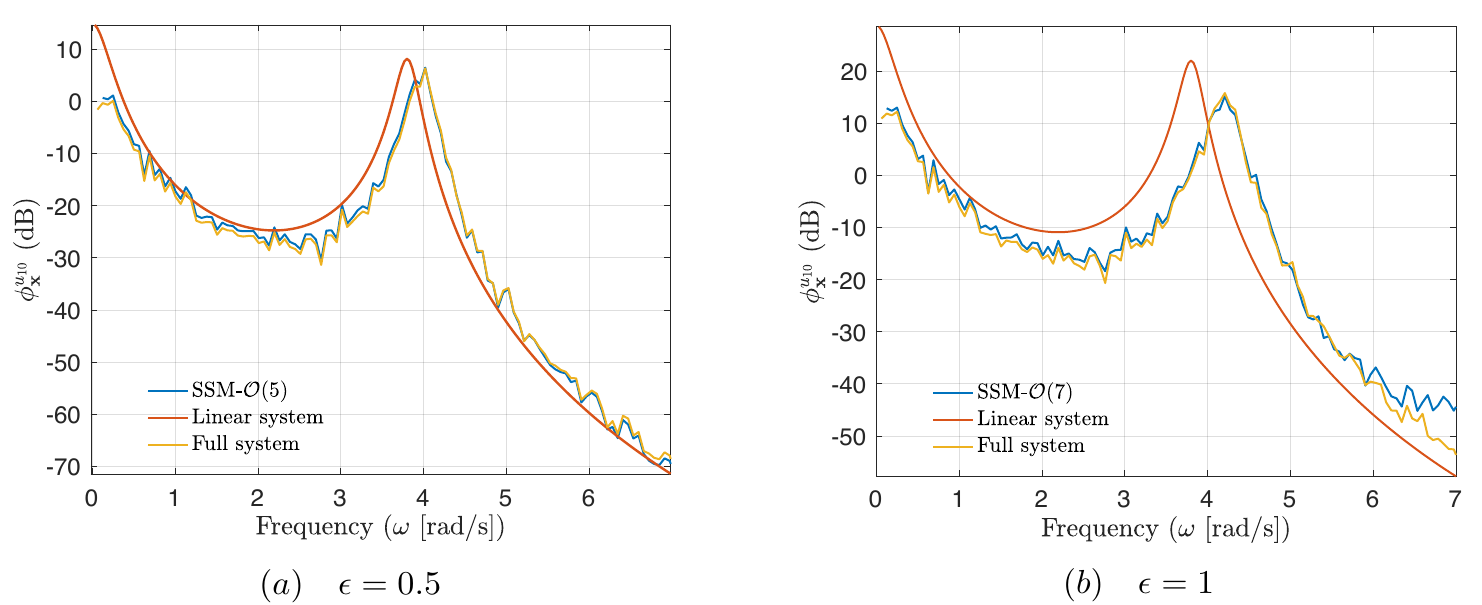} 
\par\end{centering}
\caption{PSD $\phi_{\mathbf{x}}^{u_{10}}$ of the roof displacement of the
building, from the SSM-reduced model (blue) and from the full simulation
(yellow) for two earthquake magnitudes. We also plot in red the analytic
PSD for the linear system. (a) $\epsilon=0.5$ models small ground
acceleration intensity. (b) $\epsilon=1$ models large ground acceleration
intensity. }
\label{fig:psd_building} 
\end{figure}

Our leading-order random SSM-reduced models capture the deviations
from linearity as the forcing magnitude $\epsilon$ is increased to
$1$ (see Fig. \ref{fig:psd_building}b). Reduction to the random
SSM speeds up the Monte Carlo simulations significantly. In Table
\ref{tab:building}, we show that the time for the full simulation
is nearly $7$ times the combined time taken for constructing the
SSM-reduced order model and performing the reduced Monte Carlo simulations
(just $35$ seconds).

\smallskip{}

\begin{table}[H]
\begin{centering}
\begin{tabular}{|c|c|c|}
\hline 
\multirow{1}{*}{Full system (HH:MM:SS)} & 2D SSM (HH:MM:SS)  & Number of MC simulations\tabularnewline
\hline 
00:03:22  & 00:00:35  & 50\tabularnewline
\hline 
\end{tabular}
\par\end{centering}
\caption{Run times for the full building model and its random SSM-reduced counterpart
for $50$ Monte Carlo simulations in the $\epsilon=1$ forcing regime.
All computations were performed on MATLAB version 2022b installed
on 50 nodes of the ETH Z\"urich Euler supercomputing cluster with Intel(R)
Xeon(R) CPU E3-1284L v4 @ 2.90GHz processor.}
\label{tab:building} 
\end{table}

\subsection{Random base excitation of a von K\'arm\'an beam}

\label{sec:beam}

Our previous examples focused on discrete models with localized spring
nonlinearities. We now proceed to demonstrate the applicability of
our methods to more challenging continuum beam models with distributed
nonlinearities. We use a finite element model (FEM) approximation
for the the von Kármán beam as described by \citet{jain2018} which
is based on textbook methods by \citet{reddy2014introductionNonlinearFEM}.
This model is implemented in a finite- element solver \citet{yetFem},
which yields the matrices $\mathbf{M}$, $\mathbf{C}$, $\mathbf{K}$
and the nonlinearity by $\mathbf{f}_{nl}(\mathbf{q},\mathbf{\dot{q}})$
in the form (\ref{eq:eom}) for specific parameter values and boundary
conditions. We use 20 beam elements and impose a cantilevered boundary
condition at one end to describe the continuum model, which results
in $n=60$ degrees of freedom. The position vector of the finite element
model is $\mathbf{q}=\begin{pmatrix}x_{1}, & z_{1}, & w_{1}, & \dots & x_{20}, & z_{20}, & w_{20}\end{pmatrix}^{\top}$.
Here $x$ is the axial displacement, $z$ the transverse displacement
and $w$ the deflection angle measured from the beam axis for one
element of the beam. In Fig. \ref{fig:beam}, we show the beam and
list the parameters used in our simulations. These parameter values
correspond to a slender sheet made up of aluminium.

\begin{figure}[H]
\begin{centering}
\includegraphics[width=0.6\textwidth]{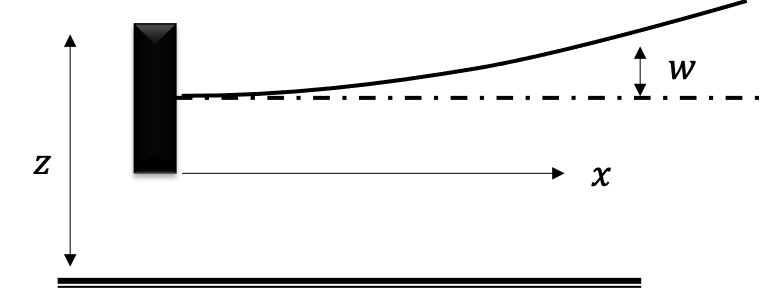} 
\par\end{centering}
\caption{Schematic view of a cantilevered beam model. The material density
is 2,700 $\text{[kg/m\ensuremath{^{3}}]}$, the Young's modulus is
$70\text{ [GPa]}$, the Poisson's ratio is 0.3, the length is $1\text{ [m]}$,
the width is $1\text{ [m]}$ and the thickness is $0.001\text{ [m]}$. }
\label{fig:beam} 
\end{figure}

We subject this beam to random base excitation, which emulates seismic
excitations applied to the cantilevered end of the beam. Accordingly,
the forcing vector takes the form $\mathbf{p}\left(\boldsymbol{\theta}^{t}\left(\boldsymbol{\nu}\right)\right)=-\epsilon\theta^{t}(\nu)\mathbf{M}\begin{pmatrix}0, & 1, & 0, & \dots & ,0, & 1, & 0\end{pmatrix}^{\top}$.
To generate the bounded random signal $\theta^{t}(\nu)$ we use Method
2 and output the random variable $\ddot{a}$ for parameter values
$m=5\text{ [kg]}$, $c=100\text{ [Ns/m]}$, $k=20\text{ [N/m]}$ from
(\ref{eq:filter}).

We monitor the transverse tip displacement of the beam for two maximal
base accelerations under the forcing parameters $\epsilon=0.5$ and
$\epsilon=1$. We find that computing an order $N=7$ random SSM-reduced
model suffices to reproduce the statistics of the full order model
for both forcing magnitudes.

In Figs. \ref{fig:psd_beam}a,b, we plot the PSD (in dB) of the tip
displacement for $m=50$ Monte Carlo simulations. Our SSM-reduced
computations show close agreement with full order model. For larger
forcing magnitudes $\epsilon=1$, the forced response is amplified
for all the models (see Fig. \ref{fig:psd_beam}b). Due to no nonlinear
damping, the linear model overestimates the amplification while the
leading-order random SSM-reduced model matches the true response accurately.

\begin{figure}[H]
\includegraphics[width=1\textwidth]{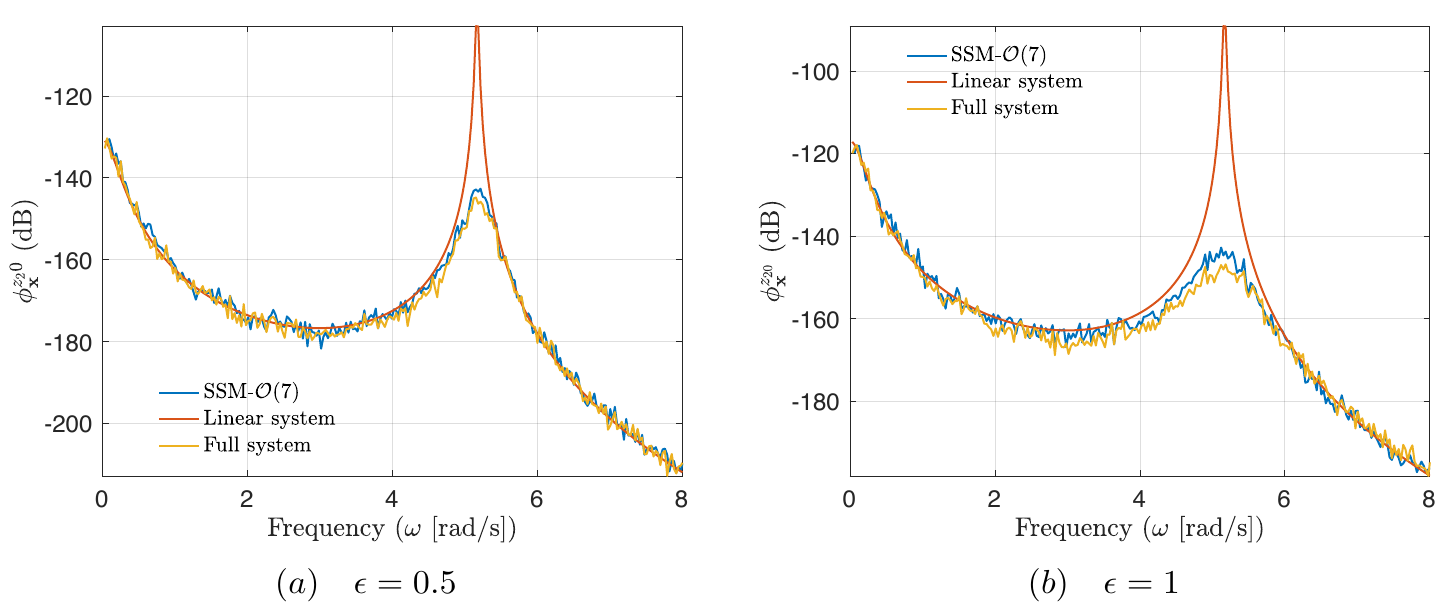}

\caption{PSD $\phi_{\mathbf{x}}^{z_{20}}$ of the transverse tip displacement
of the beam obtained from the SSM-reduced order model (blue) and from
the full order model simulation (yellow) for two earthquake magnitudes.
We also plot in red the analytic PSD for the linear system. (a) $\epsilon=0.5$
models small ground acceleration intensity (b) $\epsilon=1$ models
large ground acceleration intensity. }
\label{fig:psd_beam} 
\end{figure}

As before in Table \ref{tab:beam}, we list the run times of $m=50$
Monte Carlo experiments for the full and reduced models. The beam
problem has particularly low damping and to reach any relevant statistics
one must simulate the forcing realizations for long times. Still,
in just under a minute and a half, our random SSM-reduced order models
provide averaged forced outcomes. For comparison, the full order model
takes roughly 2 hours and 40 minutes for the MC simluations. \smallskip{}

\begin{table}[H]
\begin{centering}
\begin{tabular}{|c|c|c|}
\hline 
\multirow{1}{*}{Full system (HH:MM:SS)} & 2D SSM (HH:MM:SS)  & Number of MC simulations\tabularnewline
\hline 
02:39:42  & 00:01:26  & 50\tabularnewline
\hline 
\end{tabular}
\par\end{centering}
\caption{Run times for beam model and reduced model for $50$ Monte Carlo simulations
in the $\epsilon=1$ forcing regime. All computations were performed
on MATLAB version 2022b installed on 50 nodes of the ETH Z\"urich Euler
supercomputing cluster with Intel(R) Xeon(R) CPU E3-1284L v4 @ 2.90GHz
processor.}
\label{tab:beam} 
\end{table}

\subsection{Von K\'arm\'an plate subject to a stochastic pressure field}

\label{sec:plate} Our final example is a continuum plate model used
in previous SSM-related studies under deterministic forcing (see \citet{jain2022}
and \citet{li22a}). We modify the plate model to be flat, i.e., set
the curvature to zero, and impose simply supported boundary conditions
at the edges $A$ and $D$, as shown in Fig. \ref{fig:plate}.

\begin{figure}[H]
\begin{centering}
\includegraphics[width=0.6\textwidth]{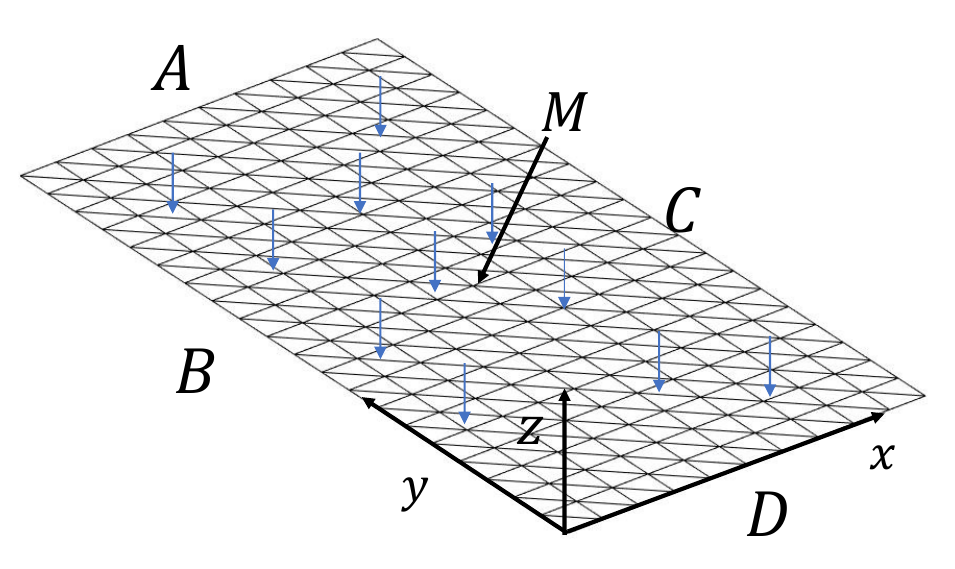} 
\par\end{centering}
\caption{Schematic view of the von Kármán plate model. The material density
is 2,700 $\text{[kg/m\ensuremath{^{3}}]}$, the Young's modulus is
$70\text{ [GPa]}$, Poisson's ratio is 0.33, the length is $2\text{ [m]}$,
the width is $1\text{ [m]}$ and the thickness $0.01\text{ [m]}$.
Blue arrows indicate the direction of the uniform body force applied.
The midpoint of the plate, denoted as $M$, is used to monitor the
output signal.}
\label{fig:plate} 
\end{figure}

The plate is made up of flat triangular shell elements with each node
on the triangle having 6 degrees of freedom. In Fig. \ref{fig:plate},
we show these elements with the caption stating the parameter values
used in our simulations. Specifically, we chose 400 elements which
amounts to 1,320 degrees of freedom. We plug in these parameter values
and element numbers in the FE solver of \citet{yetFem} to obtain
the quantites on the left-hand side of eq. (\ref{eq:eom}).

As for the external forcing, we apply a random uniform body force
in the vertical direction $\mathbf{v}$. This models external wind
pressure acting on a flat plate, where the wind velocity is the random
variable (see \citet{windSDE}). We generate the stochastic velocity
using Method 2, with the random variable $\dot{a}$ in eq. (\ref{eq:filter})
modeling the wind velocity. We set $m=5\text{ [kg]}$, $c=100\text{ [Ns/m]}$,
$k=20\text{ [N/m]}$ in eq. (\ref{eq:filter}) to obtain the realisations
for $\dot{a}$. Piecing all this together, the forcing vector is $\mathbf{p}\left(\boldsymbol{\theta}^{t}\left(\boldsymbol{\nu}\right)\right)=\epsilon c_{d}\rho[\dot{a}(t)]^{2}\mathbf{v}$.
We further set $c_{d}\rho=2\text{ [kg/m\ensuremath{^{3}}]}$, where
$c_{d}$ is the drag coefficient of the plate in air and $\rho$ is
the density of the air.

\begin{figure}[H]
\begin{centering}
\includegraphics[width=1\textwidth]{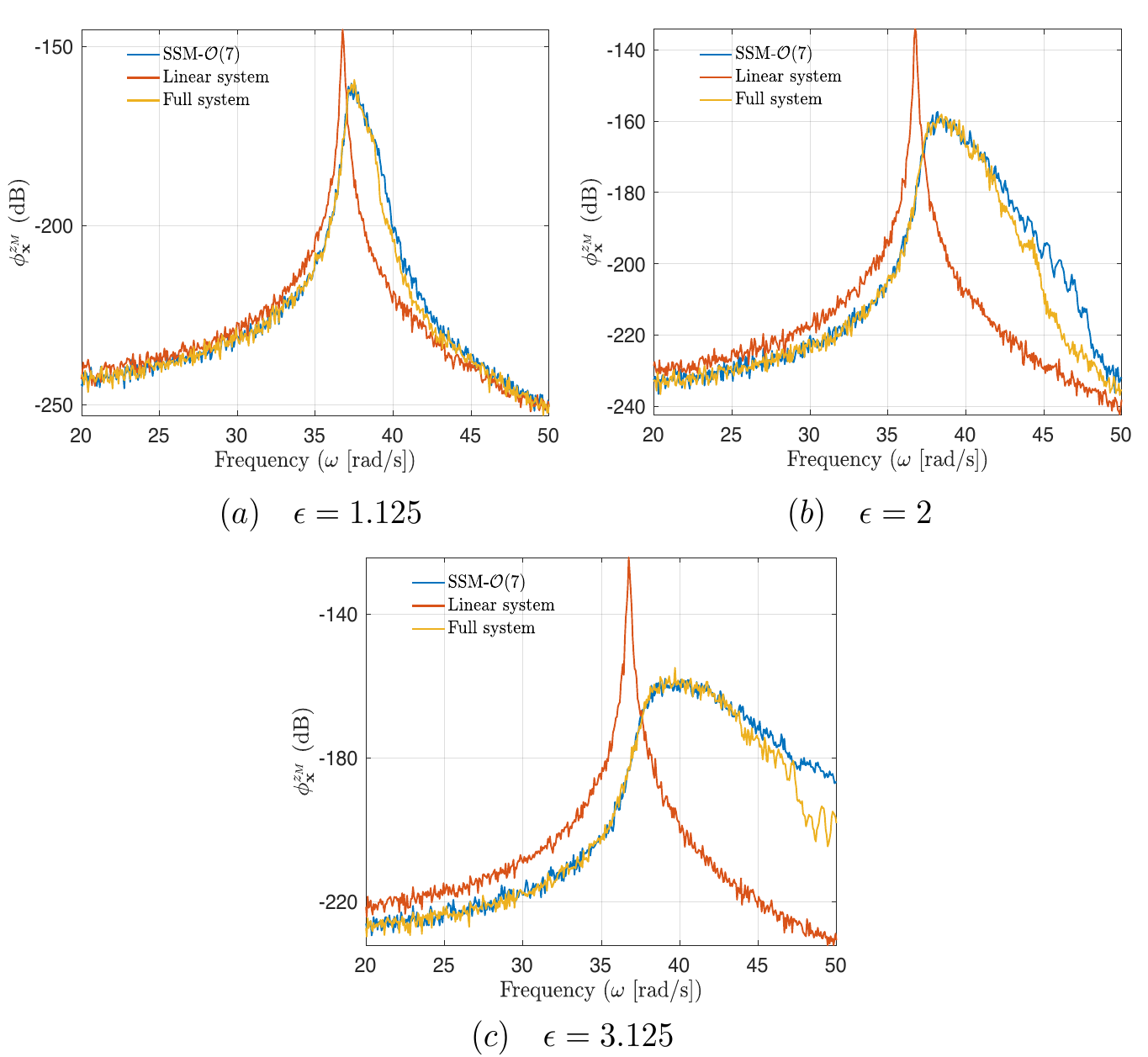} 
\par\end{centering}
\caption{PSD $\phi_{\mathbf{x}}^{z_{M}}$ of the out-off plane displacement
of the center point $M$ obtained from the SSM-reduced order model
(blue) and from the full-order model (yellow) for two earthquake magnitudes.
We also plot in red the numerical PSD for the linear system for different
$\epsilon$ values. (a) $\epsilon=1.125$ (small wind pressure fluctuations).
(b) $\epsilon=2$ (moderate wind pressure fluctuations). (c) $\epsilon=3.125$
(large wind pressure fluctuations). }
\label{fig:psd_plate} 
\end{figure}

In this example, we compute a random SSM-reduced model up to order
$N=7$. We again validate our results with the full order model for
$m=50$ Monte Carlo simulations. In Figs. \ref{fig:psd_plate}a-c,
we see the PSDs of the reduced-order models capture the trends of
the true system for increasing forcing magnitudes. These trends are
well documented in the literature for plate models (see, e.g., \citet{Reinhall1989effect}).

The run time comparison between the full and reduced models now yields
a dramatic difference. The combined time for constructing the random
SSM model and the complete Monte Carlo simulation is under 2 minutes.
For the full order model simulation it takes nearly half a day ($13$
hours). We refer to Table \ref{tab:plate} for the exact solve times.

\smallskip{}

\begin{table}[H]
\begin{centering}
\begin{tabular}{|c|c|c|}
\hline 
\multirow{1}{*}{Full system (HH:MM:SS)} & 2D SSM (HH:MM:SS)  & Number of MC simulations\tabularnewline
\hline 
12:50:40  & 00:01:45  & 50\tabularnewline
\hline 
\end{tabular}
\par\end{centering}
\caption{Run times for plate model and reduced model for $50$ Monte Carlo
simulations in the $\epsilon=3.125$ forcing regime. All computations
were performed on MATLAB version 2022b installed on 50 nodes of the
ETH Z\"urich Euler supercomputing cluster with Intel(R) Xeon(R) CPU
E3-1284L v4 @ 2.90GHz processor.}
\label{tab:plate} 
\end{table}

\section{Conclusion}

We have extended prior mathematical results on the existence and reduced
dynamics of spectral submanifolds (SSMs) from deterministic dynamical
systems to dynamical systems with additive and uniformly bounded random
noise. While the mathematical proofs of our results are valid for
the case of small external random forcing, we expect them to extend
to larger forcing values as well, as seen in the examples we presented.
This is in line with our findings for moderate deterministic aperiodic
forcing (see, e.g., \citet{haller24}).

Our main result is that appropriately defined random attracting invariant
manifolds perturb from primary attracting SSMs of the deterministic,
unforced limit of a dynamical system. These random SSMs are no longer
tangent to eigenspaces of the linearized system but continue to be
invariant under any given realization of the flow map and have the
same degree of smoothness as their unforced counterpart. The leading-order
reduced dynamics on the random SSMs consists of the reduced dynamics
of the deterministic limit of the SSM with an additive random term.
The deterministic SSM-reduced dynamics can be computed via equation-driven
or data-driven model reduction applied to the unforced system using
SSMTool (\citet{jain23}) or SSMLearn (see \citet{cenedese21}), respectively.
For state-independent random forcing, the additional random term in
the leading-order SSM-reduced dynamics is simply the projection of
the random forcing to the spectral subspace. 

Trajectories of the full system in the domain of attraction of random
SSMs synchronize exponentially fast with trajectories of the SSM-reduced
dynamics. This enables us to perform Monte-Carlo simulations directly
on the SSM-reduced dynamics to obtain the response statistics of the
full system. Such an SSM-reduced computation brings an immense speed-up
relative to simulations performed on full finite-element problems.
This is generally true for any model obtained from model-reduction
tools but SSM-reduction comes with mathematical guarantees and requires
no tuning of hyperparameters or training of neural nets. Importantly,
the deterministic core of the SSM-reduced dynamics remains the same
even if one changes the noise model. 

The leading-order SSM-reduction formulas derived in this paper only
allow us to test our results with the full order model up to a certain
magnitude of forcing. This limitation can be addressed by computing
higher-order corrections due to forcing. Such corrections have already
been computed in recent work (\citet{haller24}) and integration of
these results into the package \textit{SSMtool} is currently underway.
At this point, those higher-order corrections are already directly
applicable to any given realization of a bounded random forcing model,
as was illustrated by \citet{haller24} for chaotic forcing obtained
from numerical simulations of the classic Lorenz attractor.

In this paper, we have illustrated random SSM reduction in an equation-driven
setting where the deterministic core of the reduced model was obtained
from SSMTool. Data-driven random SSM reduction proceeds along the
same lines, with SSMLearn providing the deterministic part of the
reduced model. A continually growing body of literature is available
on the accuracy of low-dimensional deterministic models obtained from
SSMLearn applied to data generated from finite-element models or experimental
data, including experimental videos (see \citet{cenedese22a,cenedese22b,haller22,haller23,kaszas22,axas23b,liu24,xu24,cenedese24,yang24}).
Specific demonstrations of the present results in those data-driven
settings will appear elsewhere.

\appendix

\section{Proof of Theorem 1}

\label{app:A}

Following \citet{li13}, we use the following definition of normal
hyperbolicity for random invariant manifolds: 
\begin{defn}
\label{def: random NHIM}For a fixed $\epsilon\geq0$, a random invariant
manifold $\mathcal{M}\left(\boldsymbol{\nu}\right)$ of system (\ref{eq:random ODE})
is \emph{normally hyperbolic} if for almost every $\boldsymbol{\nu}\in\mathcal{V}$
and $x\in\mathcal{M}\left(\boldsymbol{\nu}\right)$, there exists
a splitting 
\[
T_{x}X=E^{u}\left(x;\boldsymbol{\nu}\right)\oplus E^{c}\left(x;\boldsymbol{\nu}\right)\oplus E^{s}\left(x;\boldsymbol{\nu}\right),\qquad E^{c}\left(x;\boldsymbol{\nu}\right)=T_{x}\mathcal{M}\left(\boldsymbol{\nu}\right),\qquad x\in\mathcal{M}\left(\boldsymbol{\nu}\right),
\]
of each tangent space $T_{x}\mathbb{R}^{n}$ along $\mathcal{M}\left(\boldsymbol{\nu}\right)$
into subspaces $E^{i}\left(x;\boldsymbol{\nu}\right)$, for $i=u,c,s$,
such that

(i) The splitting is $C^{0}$ in $x$ and measurable in $\boldsymbol{\nu}$
and has associated projections $\Pi^{i}\left(x;\boldsymbol{\nu}\right)\colon T_{x}X\to E^{i}\left(x,\boldsymbol{\nu}\right)$,
for $i=u,c,s$, such that

(ii) The splitting is invariant, i.e., 
\begin{align*}
D_{x}\mathbf{F}_{\epsilon}^{t}(x;\boldsymbol{\nu})E^{i}\left(x;\boldsymbol{\nu}\right) & =E^{i}\left(\boldsymbol{\theta}^{t}\left(\boldsymbol{\nu}\right),\mathbf{F}_{\epsilon}^{t}(x;\boldsymbol{\nu})\right),\quad i=u,s.
\end{align*}

(iii) The splitting is hyperbolic, i.e., there exist $(\boldsymbol{\theta}^{t},\mathbf{F}_{\epsilon}^{t})$-invariant
random variables $\bar{\alpha},\bar{\beta}\colon\mathcal{M}\to\left(0,\infty\right)$
with $\bar{\alpha}<\bar{\beta}$ and a tempered random variable $K\colon\mathcal{M}\to[1,\infty)$
(i.e., $\left|K\left(x;\boldsymbol{\theta}^{t}\left(\boldsymbol{\nu}\right)\right)\right|$
has sub-exponential growth as $t\to\pm\infty$ for any $x\in\mathcal{M}\left(\boldsymbol{\nu}\right)$)
such that 
\begin{align*}
\left\Vert D_{x}\mathbf{F}_{\epsilon}^{t}(x;\boldsymbol{\nu})\Pi^{s}\left(x;\boldsymbol{\nu}\right)\right\Vert  & \leq K\left(x;\boldsymbol{\nu}\right)e^{-\bar{\beta}\left(\boldsymbol{\nu},x\right)t},\quad t\geq0,\\
\left\Vert D_{x}\mathbf{F}_{\epsilon}^{t}(x;\boldsymbol{\nu})\Pi^{u}\left(x;\boldsymbol{\nu}\right)\right\Vert  & \leq K\left(x;\boldsymbol{\nu}\right)e^{\bar{\beta}\left(\boldsymbol{\nu},x\right)t},\quad t\leq0,\\
\left\Vert D_{x}\mathbf{F}_{\epsilon}^{t}(x;\boldsymbol{\nu})\Pi^{c}\left(x;\boldsymbol{\nu}\right)\right\Vert  & \leq K\left(x;\boldsymbol{\nu}\right)e^{\bar{\alpha}\left(\boldsymbol{\nu},x\right)\left|t\right|},\quad t\in\mathbb{R}.
\end{align*}
The random manifold $\mathcal{M}$ is called \emph{$\rho$-normally
hyperbolic} if $\rho\bar{\alpha}<\bar{\beta}$. 
\end{defn}
Under our assumptions in Section \ref{subsec:The-unforced-system},
$\mathbf{F}_{0}^{t}\colon X\to X$ is a class $C^{r}$ deterministic
flow map that has a class $C^{r}$, $\rho$-normally hyperbolic, deterministic
invariant manifold $\mathcal{W}_{0}(E)$ with $\rho\leq r$. (In this
case, we have $E^{u}\left(x;\boldsymbol{\nu}\right)=\emptyset$, and
$\bar{\alpha}=\alpha$ and $\bar{\beta}=\beta$ are deterministic
constants in Definition \ref{def: random NHIM}.) If $\mathcal{W}_{0}(E)$
were invariant, then the main result of \citet{li13} applied to this
setting would yield for small enough $\epsilon>0$ that $\mathbf{F}_{\epsilon}^{t}(\,\cdot\,;\boldsymbol{\nu})$
has a unique, class-$C^{r}$, compact normally hyperbolic invariant
manifold $\tilde{\mathcal{M}}$ that is $C^{1}$-close and $C^{r}$-diffeomorphic
to $\mathcal{W}_{0}(E)$ for all $\boldsymbol{\nu}\in\mathcal{V}$.
Furthermore, $\mathcal{\tilde{M}}$ would be normally attracting.
We stress that for a general stochastic differential equations, even
small perturbations can lead to large changes in the flow map over
short time. For that reason, the related random invariant manifold
results of \citet{li13} would only apply because we assumed uniformly
bounded noise.

However, $\mathcal{W}_{0}(E)$ is not an invariant (boundaryless)
manifold but an inflowing-invariant manifold with boundary, as sketched
in Fig. \ref{fig:Geometry of SSM}. While \citet{li13} also prove
related additional results for inflowing-invariant manifolds, those
results require the manifold to be normally repelling. In contrast,
$\mathcal{W}_{0}(E)$ is normally attracting, which makes those additional
results \citet{li13} inapplicable to our setting without further
considerations.

The same issue also arises in the classic persistence theory of normally
hyperbolic invariant manifolds with boundary, which is only applicable
to either normally attracting overflowing-invariant manifolds or normally
repelling inflowing-invariant manifolds. This limitation arises from
the graph transform method used in proving persistence: the advection
of nearby smooth graphs defined over the unperturbed manifolds under
the flow map needs to be a contraction mapping in forward or backward
time on the space of such graphs. Under such advection, however, candidate
graphs over normally attracting \emph{inflowing} invariant manifolds
shrink to surfaces that are no longer graphs over their full initial
domains; the same is true near normally repelling \emph{overflowing}
invariant manifolds in backward time.

To circumvent this issue and still conclude the (nonunique) persistence
of the SSM, $\mathcal{W}_{0}(E)$, of the deterministic part of the
random ODE in (\ref{eq:random ODE}) for $\epsilon>0$ small enough,
we apply the ``wormhole'' construct of \citet{eldering18} to extend
$\mathcal{W}_{0}(E)$ smoothly into a $\rho(E)$-normally hyperbolic,
class $C^{r}$, normally attracting, compact invariant manifold without
boundary. Specifically, we invoke Proposition B1 of \citet{eldering18},
which states that any inflowing-invariant, class $C^{r}$, normally
attracting, $\rho$-normally hyperbolic invariant manifold $\mathcal{M}_{0}$
can be extended smoothly so that it becomes a subset of a $\rho$-normally
hyperbolic, normally attracting, class $C^{r}$ invariant manifold
$\hat{\mathcal{M}}_{0}$ without boundary. In addition, the stable
foliation of $W^{s}\left(\mathcal{M}_{0}\right)$ coincides with that
part of the stable foliation of the extended, boundaryless, normally
attracting invariant manifold $\hat{\mathcal{M}}_{0}$. We note that
this extension is non-unique and hence any invariant manifold result
applied to $\hat{\mathcal{M}}_{0}$ will unavoidably lose uniqueness
when restricted to a statement about $\mathcal{M}_{0}$.

Applying the above random invariant manifold results of \citet{li13}
to a forward- and backward-invariant manifold extension $\hat{\mathcal{M}}_{0}$
of $\mathcal{W}_{0}(E)$, we have Definition \ref{def: random NHIM}
satisfied for $\hat{\mathcal{M}}_{0}$ with $E^{u}\left(x;\boldsymbol{\nu}\right)=\emptyset$,
$K\left(x;\boldsymbol{\nu}\right)\equiv K_{0}$, $\bar{\beta}\left(\boldsymbol{\nu},x\right)\equiv\left|\mathrm{Re}\,\lambda_{d+1}\right|-\nu$
and $\bar{\alpha}\left(\boldsymbol{\nu},x\right)\equiv\left|\mathrm{Re}\,\lambda_{1}\right|-\nu$
for some small $\nu>0$. The cited results of \citet{li13} then yield
a persisting random invariant manifold $\hat{\mathcal{M}}_{\epsilon}\left(\boldsymbol{\nu}\right)$
that is $\mathcal{O}\left(\epsilon\right)$ $C^{1}$-close and $C^{\rho(E)}$-diffeomorphic
to $\hat{\mathcal{M}}_{0}$ . Passing to a subset $\mathcal{W}_{\epsilon}(E;\boldsymbol{\nu})$
of $\hat{\mathcal{M}}_{\epsilon}(\boldsymbol{\nu})$ that is a graph
over $\mathcal{W}_{0}(E)$, we finally conclude the (non-unique) persistence
of a normally attracting, class $C^{\rho(E)}$ random, inflowing-invariant
manifold $\mathcal{W}_{\epsilon}(E;\boldsymbol{\theta}^{t}(\boldsymbol{\nu})$
that is $\mathcal{O}\left(\epsilon\right)$ $C^{1}$-close and $C^{\rho(E)}$-diffeomorphic
to $\mathcal{W}_{0}(E)$, as claimed in Theorem \ref{thm:main}. The
$\mathcal{O}\left(\epsilon\right)$ $C^{1}$-closeness of the forced
SSM to the unforced one can actually be improved to $\mathcal{O}\left(\epsilon\right)$
$C^{\rho(E)}$-closeness by adding $\epsilon$ as a dummy variable
to the variables over which the unforced SSM is defined. We finally
note that despite the non-uniqueness of the SSM, the Taylor expansion
of $\mathcal{W}_{\epsilon}(E;\boldsymbol{\theta}^{t}(\boldsymbol{\nu})$
is unique up to order $C^{\rho(E)}$ in $\epsilon$. 

\section{Spectral density for linear systems\label{app:B} }

Consider the randomly forced linear, time-invariant, multi-degree-of-freedom
system 
\[
\mathbf{M}\ddot{\mathbf{q}}+\mathbf{C}\dot{\mathbf{q}}+\mathbf{K}\mathbf{q}=\mathbf{p}(t).
\]

By taking the Fourier transform of the above equation, we obtain the
energy transfer function in the frequency domain as 
\[
\mathbf{H}(\omega)=[-\omega^{2}\mathbf{M}+i\omega\mathbf{C}+\mathbf{K}]^{-1}.
\]

If we assume the spectral density matrix of the random forcing $\mathbf{p}(t)$
is $\mathbf{\Phi}_{\mathbf{p}}$, we can explicitly calculate the
spectral density matrix for $\mathbf{q}$ as 
\begin{equation}
\mathbf{\Phi}_{\mathbf{q}}(\omega)=\mathbf{H}(\omega)\mathbf{\Phi}_{\mathbf{p}}(\omega)\mathbf{H}^{\star}(\omega).\label{eq:linear_psd}
\end{equation}
When the system is one-dimensional, this relationship simplifies to
$\mathbf{\Phi}_{\mathbf{q}}(\omega)=\|\mathbf{H}(\omega)\|^{2}\mathbf{\Phi}_{\mathbf{g}}(\omega)$.
For the example in Section \ref{sec:Suspension}, we use the above
relation to analytically compute the output response when the system
is assumed to be linear. For the example, in Section \ref{sec:building},
where we apply a reflective boundary for the It\^o process, no analytic
relationship of the type (\ref{eq:linear_psd}) is available for the
linear system, but one can still use this formula as an approximation
for the linear response. This also applies to the example in Section
\ref{sec:beam}, where we sample from a truncated Gaussian distribution.
In this case, although the linear relation is not exact, it serves
as a good estimate to capture the overall trend. Lastly, for the example
in Section \ref{sec:plate}, not even an approximate analytic relationship
is available for the spectrum of the linear response, which prompts
us to use Monte Carlo simulations to obtain the linear response.

\bibliographystyle{plainnat}
\bibliography{random_SSM_bib_final}

\end{document}